\documentclass[11pt,a4paper,reqno]{amsart}
\usepackage[english]{babel}
\usepackage[T1]{fontenc}
\usepackage{verbatim}
\usepackage{palatino}
\usepackage{amsmath}
\usepackage{mathabx}
\usepackage{amssymb}
\usepackage{amsthm}
\usepackage{amsfonts}
\usepackage{graphicx}
\usepackage{esint}
\usepackage{color}
\usepackage{mathtools}
\usepackage{overpic}

\usepackage[colorlinks = true, citecolor = black]{hyperref}
\newcommand{\R}{\mathbb{R}}

\newcommand{\N}{\mathbb{N}}

\newcommand{\calL}{\mathcal{L}}

\newcommand{\calH}{\mathcal{H}}

\newcommand{\calE}{\mathcal{E}}

\newcommand{\calU}{\mathcal{U}}

\newcommand{\spt}{\operatorname{spt}}
\newcommand{\Hd}{\dim_{\mathrm{H}}}

\newcommand{\diam}{\operatorname{diam}}

\newcommand{\dist}{\operatorname{dist}}

\def\fz{\infty}
\def\az{\alpha}

\def\dist{{\mathop\mathrm{\,dist\,}}}

\def\lz{\lambda}
\def\dz{\delta}

\def\ez{\epsilon}
\def\kz{\kappa}

\def\gz{{\gamma}}

\def\pa{\partial}

\def\Barint_#1{\mathchoice
          {\mathop{\vrule width 6pt height 3 pt depth -2.5pt
                  \kern -8pt \intop}\nolimits_{#1}}%
          {\mathop{\vrule width 5pt height 3 pt depth -2.6pt
                  \kern -6pt \intop}\nolimits_{#1}}%
          {\mathop{\vrule width 5pt height 3 pt depth -2.6pt
                  \kern -6pt \intop}\nolimits_{#1}}%
          {\mathop{\vrule width 5pt height 3 pt depth -2.6pt
                  \kern -6pt \intop}\nolimits_{#1}}}

\numberwithin{equation}{section}

\theoremstyle{plain}
\newtheorem{thm}[equation]{Theorem}
\newtheorem*{"thm"}{"Theorem"}

\newtheorem{lemma}[equation]{Lemma}

\newtheorem{ex}[equation]{Example}
\newtheorem{cor}[equation]{Corollary}

\theoremstyle{definition}

\newtheorem{definition}[equation]{Definition}

\theoremstyle{remark}
\newtheorem{remark}[equation]{Remark}

\newcommand{\nref}[1]{(\hyperref[#1]{#1})}

\DeclareMathSymbol{\intop}  {\mathop}{mathx}{"B3}
\author{Jiayin Liu}
\title[Restricted Projections in $\R^{n+1}$]{Restricted Projections to Lines in $\R^{n+1}$}
\address{Department of Mathematics and Statistics\\ University of Jyv\"askyl\"a,
P.O. Box 35 (MaD)\\
FI-40014 University of Jyv\"askyl\"a\\
Finland}
\email{jiayin.mat.liu@jyu.fi}
\date{\today}
\thanks{The author is supported by the Academy of Finland via
the projects  \emph{Incidences on Fractals}, grant No. 321896
 and \emph{Singular integrals, harmonic functions,
and boundary regularity in Heisenberg groups}, grant No. 321696.}

\begin{document}

\begin{abstract}
We prove the following restricted projection theorem.
Let $n\ge 3$ and $\Sigma \subset S^{n}$ be an $(n-1)$-dimensional $C^2$ manifold such that $\Sigma$ has sectional curvature $>1$. Let $Z \subset \R^{n+1}$ be analytic and let $0 < s < \min\{\dim Z, 1\}$. Then
\begin{equation*}
  \dim \{z \in \Sigma : \dim (Z \cdot z) < s\} \le (n-2)+s = (n-1) + (s-1) < n-1.
\end{equation*}
In particular,
for almost every $z \in \Sigma$, $\dim (Z \cdot z) = \min\{\dim Z, 1\}$.

The core idea, originated from K\"{a}enm\"{a}ki-Orponen-Venieri, is to transfer the restricted projection problem to the study of the dimension lower bound of Furstenberg sets of cinematic family contained in $C^2([0,1]^{n-1})$. This cinematic family of functions with multivariables are extensions of those of one variable by Pramanik-Yang-Zahl and Sogge. Since the Furstenberg sets of cinematic family contain the affine Furstenberg sets as a special case, the dimension lower bound of Furstenberg sets improves the one by H\'{e}ra, H\'{e}ra-Keleti-M\'{a}th\'{e} and D{\k{a}}browski-Orponen-Villa.

Moreover, our method to show the restricted projection theorem can also give a new proof for the Mattila's projection theorem in $\R^n$ with $n \ge 3$.
\end{abstract}

\maketitle
\tableofcontents

\section{Introduction}

We begin with the following fundamental projection theorem.
\begin{thm}\label{fundproj}
  Let $n\ge 2$, $Z \subset \R^{n}$ be Borel, and $0 \le s < \Hd Z$. Then $\Hd (Z \cdot z) = \min\{\Hd Z,1\}$ for $\calH^{n-1}$-almost every $z \in S^{n-1}$ and moreover,
  \begin{equation}\label{Kauf1}
    \Hd \{z \in S^{n-1} : \Hd (Z \cdot z) < s\} \le (n-2)+s.
  \end{equation}
  In particular,
for almost every $z \in S^{n-1}$, $\Hd (Z \cdot z) = \min\{\Hd Z, 1\}$.
\end{thm}
In the case $n=2$, Theorem \ref{fundproj} is a combination of results in \cite{MR0063439,Ka} by Marstrand and Kaufman. For all $n\ge 2$, Theorem \ref{fundproj} is proved by
Mattila \cite{MR0409774}.

Noting that $Z \cdot z$ has the same dimension as the orthogonal projection of $Z$ to span$\{z\}$, the above theorem states that for $\calH^{n-1}$-almost all directions $z \in S^{n-1}$, the orthogonal projection of $Z$ has dimension $\min\{\Hd Z,1\}$.
Since when $n\ge 3$, there are $n-1\ge 2$ dimensional many directions, it is natural to study projection properties for some restricted family of directions. Recently, there are a series of results establishing projection theorems to lines restricted to one dimensional family of directions in $\R^n$ for $n\ge 3$.

When $n=3$, Pramanik-Yang-Zahl \cite{2022arXiv220702259P} proved the following theorem, confirming a conjecture in \cite{FasslerOrponen14} by F\"{a}ssler and Orponen.
\begin{thm}\label{pyz}
  Let $I \subset \R$ be a compact interval and $\gz:I\to S^{2}$ be a $C^2$ curve that satisfies the ``non-degenerate'' condition:
\begin{equation}\label{nondeg0}
    \text{span}\{ \gz(t), \dot \gz(t), \ddot \gz(t)\}=\R^3 \text{  for all }  t \in I.
\end{equation}
Let $Z \subset \R^{3}$ be an analytic set. Then for almost every $t \in I$, $\Hd (\gz(t) \cdot Z) = \min\{\Hd Z, 1\}$ and
\begin{equation}\label{Kauf2}
    \Hd \{t \in I : \Hd (Z \cdot \gz(t)) < s\} \le s.
  \end{equation}
\end{thm}

Before Theorem \ref{pyz}, partial progress was made by F\"{a}ssler-Orponen \cite{FasslerOrponen14}, He \cite{MR4148151} and K\"{a}enm\"{a}ki-Orponen-Venieri \cite{marstrandtype}.
When $n > 3$, unlike the case in Theorem \ref{pyz} where the intersection of the one dimensional family of directions  and $S^2$ forms a $C^2$-curve in $S^2$, Gan-Guo-Wang \cite{gan2022restricted} showed the following theorem using decoupling.
\begin{thm}\label{ggw}
  Let $I \subset \R$ be a compact interval and $\gz:I\to \R^{n}$ be a $C^\infty$ curve that satisfies:
\begin{equation}\label{nondeg1}
    \text{span}\{ \gz(t), \dot \gz(t), \ddot \gz(t),\cdots,  \gz^{(n-1)}(t)\}=\R^n \text{  for all }  t \in I.
\end{equation}
Let $Z \subset \R^{n}$ be an analytic set. Then for almost every $t \in I$, $\Hd (\gz(t) \cdot Z) = \min\{\Hd Z, 1\}$.
\end{thm}
After \cite{gan2022restricted}, Zahl \cite{zahl2023maximal} reproved Theorem \ref{ggw} and in addition established the same exceptional estimate as in \eqref{Kauf2} for Theorem \ref{ggw}.

In this paper, we study the projection theorem to a class of $(n-2)$-dimensional families of lines in $\R^n$ for $n > 3$.
The main result is as follows.
\begin{thm}\label{mainproj}
Let $n\ge 3$ and $\Sigma \subset S^{n}$ be an $(n-1)$-dimensional $C^2$ manifold such that $\Sigma$ has sectional curvature $>1$ everywhere. Let $Z \subset \R^{n+1}$ be analytic and let $0 < s < \min\{\Hd Z, 1\}$. Then
\begin{equation*}
  \Hd \{z \in \Sigma : \Hd (Z \cdot z) < s\} \le (n-2)+s.
\end{equation*}
In particular,
for almost every $z \in \Sigma$, $\Hd (Z \cdot z) = \min\{\Hd Z, 1\}$.
\end{thm}
Since the regularity assumption of the $(n-1)$-dimensional manifold $\Sigma$ is $C^2$ in Theorem \ref{mainproj}, $\Sigma$ may not be foliated by an $(n-2)$-dimensional family of $C^\fz$ curves. Thus Theorem \ref{ggw} does not imply Theorem \ref{mainproj}. We also remark that our technique only works for $n \ge 3$ in Theorem \ref{mainproj} (we refer to the proof of Lemma \ref{dznbhd1}), thus Theorem \ref{mainproj} does not contain Theorem \ref{pyz} by Pramanik-Yang-Zahl as a special case.

\begin{ex} \label{eg1}\rm
   The typical examples of $\Sigma$ in Theorem \ref{mainproj} are $\Sigma_c: = S^{n} \cap \{x_{n+1}=c\}$ for $c \in (0,1)\cup(-1,0)$. Indeed, $\Sigma_c$ is an $(n-1)$-dimensional sphere in $\R^{n+1}$ with radius $\sqrt{1-c^2}$ and thus has constant sectional curvature $1/(1-c^2) >1$.
\end{ex}

Our method to show Theorem \ref{mainproj} can also give a new proof for the Mattila's projection theorem, i.e. Theorem \ref{fundproj} in $\R^n$ for $n\ge 3$. Indeed, we will show the following more general version of Theorem \ref{fundproj}.

\begin{thm}\label{mattila}
Let $n\ge 3$, and $\Sigma \subset S^{n-1}$ be an $(n-1)$-dimensional manifold. Let $Z \subset \R^{n}$ be analytic and $0 < s < \min\{\Hd Z, 1\}$. Then
\begin{equation*}
  \Hd \{z \in \Sigma : \Hd (Z \cdot z) < s\} \le (n-2)+s.
\end{equation*}
In particular,
for almost every $z \in \Sigma$, $\Hd (Z \cdot z) = \min\{\Hd Z, 1\}$.
\end{thm}

Obviously, by taking $\Sigma=S^{n-1}$ in Theorem \ref{mattila}, we obtain  Theorem \ref{fundproj}.
Moreover, one may observe that since in $\R^n$ with $n\ge 4$, $S^{n-1}$ can be foliated by the family $\{\Sigma_c: = S^{n-1} \cap \{x_{n}=c\}\}_{c \in [-1,1]}$,
according to Example \ref{eg1}, applying Theorem \ref{mainproj} to each $\Sigma_c$ (except $c=-1,0,1$), we will deduce Theorem \ref{fundproj} from Theorem \ref{mainproj}. However, this way could only show Theorem \ref{fundproj} for $\R^n$ with $n\ge 4$. Indeed, by a more intrinsic observation, in this paper, we will use the method to show Theorem \ref{mainproj} to prove Theorem \ref{mattila}, hence Theorem  \ref{fundproj} in $\R^n$ for all $n\ge 3$.

\begin{remark} \label{sec}\rm
  We remark that by Gauss' equation, for any $\Sigma \subset S^n$, we have $$ K_{\xi_i,\xi_j}(x) = \kz_{i}(x) \kz_{j}(x) + 1 ,\quad x \in \Sigma  $$
  where $\{\xi_i\}_{1 \le i\le n-1} \subset T_x\Sigma$ are the $n-1$ principal directions,
  $K_{\xi_i,\xi_j}(x)$ is the sectional curvature of $\Sigma$ at $x$ spanned by $\xi_i$ and $\xi_j$ and $\kz_{i}(x)$ is the principal curvature in the direction $\xi_i$. Thus the sectional curvature condition in Theorem \ref{mainproj} is equivalent to that the $n-1$ principal curvatures at each $x \in \Sigma$, $\kappa_1(x), \cdots, \kappa_{n-1}(x)$ are non-vanishing and have same sign.
\end{remark}

Regarding Remark \ref{sec},
the curvature condition in Theorem \ref{mainproj} can be seen as a natural generalization of Theorem \ref{pyz}. Indeed, the ``non-degenrate'' condition \eqref{nondeg0} of $\gz$ is equivalent to that $\gz$ has non-vanishing geodesic curvature $k_g$. We explain this below.
Let $\gz \subset S^2$ satisfy \eqref{nondeg0}. We assume $\gz$ is parameterised by arclength, i.e. $|\dot \gz| \equiv 1$ for simplicity.
For each $t \in I$, let $\nu =\nu(t) \in T_{\gz(t)}S^2$ be the unique (up to $\pm1$) normal of $\gz$, i.e. $\dot\gz(t) \perp \nu$.
 Noting that $\gz(t) \perp T_{\gz(t)}S^2$, we know $\{\gz(t),\dot \gz(t), \nu\}$ forms an orthonormal basis of $\R^3$. Since $|\dot \gz| \equiv 1$ implies that
$\dot \gz(t) \perp \ddot \gz(t)$ for all $t \in I$, we know
$\ddot \gz(t) = k_n(t)\gz(t) + k_g(t)\nu$ for a unique pair $k_n(t), k_g(t) \in \R$ where $k_n(t)$ is the normal curvature and $k_g(t)$ is the geodesic curvature at $\gz(t)$.
By condition \eqref{nondeg0}, we know $k_g(t) \ne 0$. For $n\ge3$ and any hypersurface in $S^n$, the corresponding notion of $k_g$ naturally generalizes to principle curvatures.

Inspired by the above relation between ``non-degenerate'' condition and curvatures, in higher dimensions, we also introduce a ``non-degenerate'' condition corresponding to the curvature condition in Theorem \ref{mainproj}.

\begin{definition}[``Non-degenerate'' Condition for $(n-1)$-dimensional Manifolds in $\R^{n+1}$] \label{nond1}
  Let $n\ge 3$ and $\Sigma \subset \R^{n+1}$ be an $(n-1)$-dimensional $C^2$ manifold. We say $\Sigma$ satisfies the ``non-degenerate'' condition if
  \begin{enumerate}
    \item [(i)]there exists a $C^2$ vector field $\nu: \Sigma \to \R^{n+1}$ such that
        $$ \nu(x)\perp x , \ \nu(x) \perp T_{x}\Sigma   \text{ and span}\{ x, T_{x}\Sigma , \nu(x)\}=\R^{n+1} \text{  for all }  x \in \Sigma; $$
    \item [(ii)]for any $x \in \Sigma$ and any pair of $C^2$-curves $\gz_1,\gz_2: [-1,1] \to \Sigma$ such that $\gz_i(0)=x$ and $\dot\gz_i(t) \ne 0$ for all $t \in [-1,1]$ and $i=1,2$,  it holds
  $$   \langle\ddot\gz_1(0),\nu(x)\rangle \langle\ddot\gz_2(0),\nu(x)\rangle >0 .  $$
  \end{enumerate}

\end{definition}

We will prove the following restricted projection theorem for ``non-degenerate'' manifolds.

\begin{thm}\label{mainproj2}
Let $n\ge 3$ and $\Sigma \subset \R^{n+1}$ be an $(n-1)$-dimensional $C^2$ manifold satisfying the ``non-degenerate'' condition. Let $Z \subset \R^{n+1}$ be analytic and let $0 < s < \min\{\Hd Z, 1\}$. Then
\begin{equation}\label{Kauf}
  \Hd \{z \in \Sigma : \Hd (Z \cdot z) < s\} \le (n-2)+s.
\end{equation}
In particular,
for almost every $z \in \Sigma$, $\Hd (Z \cdot z) = \min\{\Hd Z, 1\}$.
\end{thm}
We will show that Theorem \ref{mainproj2} includes Theorem \ref{mainproj} as a special case and moreover, Theorem \ref{mainproj2} also contains the result of \cite[Theorem 1.1]{ohm2023projection} as a special case. Hence the rest of the paper is devoted to showing Theorem \ref{mainproj2}.
We outline the ideas of the proof.

\subsection{Restricted projections, cinematic families and Furstenberg sets}
Let $Z$ and $\Sigma$ be as in Theorem \ref{mainproj2}. By proving the theorem on each chart of $\overline \Sigma$ (the closure of $\Sigma$), we can assume $\Sigma \subset \R^{n+1}$ as a $C^2$ diffeomorphism from $[0,1]^{n-1}$ to $\R^{n+1}$ and therefore throughout the paper we assume $\Sigma \in C^2([0,1]^{n-1}, \R^{n+1})$ to be a $C^2$ diffeomorphism.
For each $z \in Z$, define
$$  f_z(x):=\langle \Sigma(x) , z   \rangle  , \quad x \in [0,1]^{n-1}. $$
Then we obtain a family $F=\{f_z \}_{z \in Z} \subset C^2([0,1]^{n-1})$.
Note that for each $x \in [0,1]^{n-1}$, the intersection
$$ \bigcup_{z \in Z} [(\{x\} \times \R) \cap \mbox{graph}(f_z)]= [\{x\} \times \R] \cap [\bigcup_{z \in Z} (x,f_z(x))] \subset \R^n $$
is the projection of $Z$ to span$\{\Sigma(x)\}$ (up to scaling), i.e. $\langle \Sigma(x) , Z   \rangle$. Thus the study of the $(n-1)$-dimensional projections can be transferred to the study of the dimension lower bound of the union of
  graphs of the functions in $F$. Correspondingly, the proof of exceptional estimate \eqref{Kauf} can also be transferred to that of the dimension lower bound of the union of
  certain subsets contained in the graphs of the functions in $F$.
 This idea was first discovered by K\"{a}enm\"{a}ki-Orponen-Venieri \cite{marstrandtype} to study restricted projections in $\R^3$ and then employed by Pramanik-Yang-Zahl to show Theorem \ref{pyz}.

We depict the heuristic idea of the proof in more details. Fix $0<s<s' \le \Hd Z$ and choose any subset $X \subset [0,1]^{n-1}$ with $\Hd X = n-2+ s'$.
Then the fact that Theorem \ref{mainproj2} holds is equivalent to that
\begin{equation}\label{heu1}
  \Hd (\langle \Sigma(x) , Z   \rangle) = \Hd (\bigcup_{z \in Z} [(\{x\} \times \R) \cap \mbox{graph}(f_z)]) \ge s, \quad \calH^{n-2+s'}\mbox{-a.e. } x \in X.
\end{equation}
Consider the set
\begin{equation}\label{heu3}
    E: =\bigcup_{x \in X}\bigcup_{z \in Z} [(\{x\} \times \R) \cap \mbox{graph}(f_z)] = \bigcup_{z \in Z} [(X \times \R) \cap \mbox{graph}(f_z)] \subset \R^{n}.
\end{equation}
We will further verify \eqref{heu1} by showing that $E$ has dimension lower bound
\begin{equation}\label{heu2}
  \Hd E \ge \Hd X +s = n-2+s' +s.
\end{equation}
Note that there is no Fubini's theorem for Hausdorff measures, so instead, we will use a standard $\dz$-discretization argument to show \eqref{heu2} implies \eqref{heu1}. This is done in Section 5.

To see \eqref{heu2},
note that $E$ can be seen as a union of $(n-2+s')$-dimensional subsets from a family of graphs of functions. In general, \eqref{heu2} can fail for an arbitrary family of graphs of functions. However,
under the ``non-degenerate'' condition for $\Sigma$, we will show that $F=\{f_z\}_{z \in Z}$ forms a cinematic family in the following sense and \eqref{heu2} holds for graphs of a cinematic family.

\begin{definition}[Cinematic Family]\label{cine}
  Let $F \subset C^2(U)$ where $U \subset \R^k$ is a domain, $k \in \mathbb{N}$. We say $F$ is a cinematic family of
functions, with cinematic constant $K$, doubling constant $D$ and modulus of continuity $\az$ if the following conditions
hold.
\begin{enumerate}
  \item $F$ is contained in a ball of diameter $K$ under the usual metric on $C^2(U)$.
  \item $F$ is a doubling metric space, with doubling constant at most $D$.
  \item For all $f, g \in F$ and all $(x,\xi) \in U \times S^{k-1}$, we have
  $$  \inf_{x \in U} \{ |f(x) - g(x)| + |\nabla f(x) -\nabla g(x)| + |\nabla_\xi\nabla_\xi f(x) - \nabla_\xi\nabla_\xi g(x)| \} \ge K^{-1} \|f-g\|_{C^2(U)}.  $$
  \item There exists an increasing function $\az \in C^0([0,1], [0,+\fz))$ with $\az(0)=0$ and $0<\az(s)\le K^{-1}s$ for all $s \in [0,1]$ such that for all $f, g \in F$ and all $\xi \in S^{k-1}$, we have for any $\eta>0$ sufficiently small and $x,y \in U$ with $|x-y|\le \az(\eta)$,
  $$   |\nabla_\xi\nabla_\xi (f-g)(x) - \nabla_\xi\nabla_\xi (f-g)(y)| \le \eta. $$
\end{enumerate}
We call $\az$ as the modulus of continuity of $F$ (which is actually  the inverse function of the standard modulus of continuity).
\end{definition}


Definition \ref{cine} can be seen as a natural generalization to functions of multi-variables from a cinematic family of one variable in \cite[Definition 1.6]{2022arXiv220702259P} and furthermore in \cite{MR1098614} by Sogge where the name cinematic comes from. We also remark that the condition (4) in Definition \ref{cine} is not required in the definition for cinematic family of one variable in \cite{2022arXiv220702259P} (see the paragraph after Proposition 2.1 therein). This is because the paper \cite{2022arXiv220702259P} focuses on the specific cinematic family arising from the restricted projection which naturally enjoys condition (4).

With the notion of cinematic family, we define Furstenberg sets of hypersurfaces, which contains the set $E$ in \eqref{heu3} as a special case.

\begin{definition}[Furstenberg sets of hypersurfaces] \label{furcine}
    Let $s,t>0$ and $F \subset C^2([0,1]^{n-1},[0,1])$ be a cinematic family with
cinematic constant $K$, doubling constant $D$ and modulus of continuity $\az$. A set $E \subset \R^{n}$ is called a $(s,t)_{n-1}$-Furstenberg set of
hypersurfaces with parameter set $F$ if
$\Hd (F ) \ge t $ and $\Hd (E \cap \mbox{graph} (f)) \ge s$ for each $f \in F$. Here, $\Hd (F )$ is the Hausdorff dimension with respect to the distance induced by the $C^2$-norm in $C^2([0,1]^{n-1},[0,1])$.
\end{definition}

Definition \ref{furcine} can be seen as a higher dimensional extension of circular Furstenberg sets studied in \cite{2022arXiv220401770L,fassler2023hausdorff} and Furstenberg sets of curves studied in \cite{zahl2023maximal}. The reason we call $E$ a Furstenberg set of
hypersurfaces is because each graph of a function in the parameter set $F$ is a hypersurface in $\R^n$.

Definition \ref{furcine} is also a generalization of the affine Furstenberg-type sets studied by H\'{e}ra \cite{MR3973547}, H\'{e}ra-Keleti-M\'{a}th\'{e} \cite{MR4002667} and D{\k{a}}browski-Orponen-Villa \cite{MR4452675} where the graphs from the parameter set $F$ are hyperplanes, i.e. $F$ consists of affine functions. Indeed, it is not hard to check that if $f,g \in C^2([0,1]^{n-1},[0,1])$ are two affine functions, then
$$  \inf_{x \in [0,1]^{n-1}} \{ |f(x) - g(x)| + |\nabla f(x) -\nabla g(x)| \} \ge K^{-1} \|f-g\|_{C^2([0,1]^{n-1})}, $$
for some $K>1$, which shows the family $F$ of affine functions are a cinematic family.
In Section 4, we will show the following dimension bound for  Furstenberg sets of hypersurfaces.

\begin{thm} \label{dimf}
  Let $n\ge3$ and $0<s,t \le 1$. Then every $(n-2+s,t)_{n-1}$-Furstenberg set $E$ satisfies
  \begin{equation}\label{sharp}
    \Hd (E) \ge n-2+s+\min\{t, s \}.
  \end{equation}
\end{thm}

By verifying $\Hd F = \Hd Z$, we know
 $E$ defined in \eqref{heu3} is a $(n-2+s',\Hd Z)_{n-1}$-Furstenberg set of hypersurfaces. Thus Theorem \ref{mainproj2} is implied by Theorem \ref{dimf} (or more precisely, Theorem \ref{thmconfig}, the $\dz$-discretized version of Theorem \ref{dimf}).

Since the Furstenberg sets $E$ in Theorem \ref{dimf} contains those considered in \cite{MR3973547,MR4002667,MR4452675}, the bound in Theorem \ref{dimf} improves the result of these papers in the corresponding range of $s$ and $t$. Also, it is straightforward to see that the bound in Theorem \ref{dimf} is sharp if $0 \le t \le s \le 1$. For instance, the equality in \eqref{sharp} can be achieved when the graphs from $F$ are parallel hyperplanes in $\R^n$.

\subsection{$\dz$-discretized proof}
We will reduce the proof of Theorem \ref{dimf} to a $\dz$-discretized argument. We introduce several definitions.

\begin{definition}[$(\dz,s)$-set] Let $s \geq 0$, $C > 0$, $\dz> 0$ and $(X,d)$ be a metric space. A bounded set $P \subset X$ is called a \emph{$(\delta,s,C)$-set} if
\begin{displaymath} |P \cap B(x,r)|_{\delta} \leq Cr^{s}|P|_{\delta}, \qquad x \in X, \, r \geq \delta. \end{displaymath}
Here, and in the sequel, $|E|_{\delta}$ refers to the $\dz$ covering number of $E$.  \end{definition}

The following observations are useful to keep in mind about $(\delta,s,C)$-sets. First, if $P$ is a non-empty $(\delta,s,C)$-set, then $|P|_{\delta} \geq C^{-1}\delta^{-s}$. This follows by applying the defining condition at scale $r = \delta$. Second, a $(\delta,s,C)$-set is a $(\delta,t,C)$-set for all $0 \leq t \leq s$.

The next definition can be considered as the $\dz$-discretized version of Furstenberg sets of hypersurfaces.
In the following, for all $(f,x) \in C^2([0,1]^{n-1},[0,1]) \times \R^{n}$, $\pi_{1} \colon C^2([0,1]^{n-1},[0,1]) \times \R^{n} \to C^2([0,1]^{n-1}) $ stands for the map $\pi_{1}(f,x) = f$.

\begin{definition}\label{d:config}
Let $s,t \in (0,1]$, $C>0$, and $\dz> 0$. A \emph{$(\delta,s,t,C)_{n-1}$-configuration} is a set $\Omega \subset C^2([0,1]^{n-1},[0,1]) \times \R^{n}$ such that $F := \pi_{1}(\Omega)$ is a non-empty $(\delta,t,C)$-subset of $C^2([0,1]^{n-1},[0,1])$ consisting of a cinematic family of functions with
cinematic constant $K$ and doubling constant $D$, and $E(f) := \{x \in \R^{n} : (f,x) \in \Omega\}$ is a non-empty $(\delta,n-2+s,C)$-subset of $\mbox{graph}(f)$ for all $f \in F$. Additionally, we require that the sets $E(f)$ have constant $\dz$-covering number: there exists $M \geq 1$ such that $|E(f)|_\dz = M$ for all $f \in F$.
If the constant $M$ is worth emphasising, we will call $\Omega$ a $(\delta,s,t,C,M)_{n-1}$-configuration.
\end{definition}

We will use the next theorem, in a standard way, to show Theorem \ref{dimf} and Theorem \ref{mainproj2}.
\begin{thm}\label{thmconfig}
For $0<t \le s \le 1$, there exist $\epsilon = \ez(s,t,n,K),\delta_{0}=\delta_{0}(\ez) \in (0,\tfrac{1}{2}]$ such that the following holds for all $\delta \in (0,\delta_{0}]$. Let $\Omega$ be a $(\delta,s,t,\delta^{-\epsilon},M)_{n-1}$-configuration. Then, $|\mathcal{E}|_{\delta} \geq \delta^{16\ez - t- (n-2+s)}$, where
\begin{displaymath} \mathcal{E} := \bigcup_{f \in F} E(f). \end{displaymath}
\end{thm}

For arbitrary $\dz>0$ and
subdomain $D \subset [0,1]^{n-1}$, we
define the \emph{vertical $\dz$-neighborhood}
   $$f^{\dz}(D):= \left\{(x,y)=(x_1,\cdots,x_{n-1},y) \in D\times\mathbb{R} : f(x)- \dz\le y \le f(x)+\dz\right \}.$$
   When $D=[0,1]^{n-1}$, we write $f^\dz=f^\dz([0,1]^{n-1})$ for short.
The idea to show Theorem \ref{thmconfig} is to upper bound the measure of the intersection $f^\dz \cap g^\dz$ for any $f,g \in F$.
By the cinematic property, we can show
\begin{lemma} \label{dznbhd1}
   Let $n\ge 3$, $\dz>0$ and
$F \subset C^{2}([0,1]^{n-1})$ be a cinematic family. Then for any pair $f,g \in F$ with $\|f-g\|_{C^2([0,1]^{n-1})} \in [\dz, 1]$, $f^\delta \cap g^\delta$ has measure
\begin{equation}\label{meas}
  \calL^n(f^\delta \cap g^\delta) \lesssim \frac{\delta^2}{\|f-g\|_{C^2([0,1]^{n-1})}}.
\end{equation}
\end{lemma}
Here and throughout the paper, $\calL^k$ denotes the $k$-dim Lebesgue measure in $\R^k$, $k \in \N$.
A direct consequence is the following functional estimate.
\begin{thm}\label{maxfcn}
Let $n \ge 3$.
For every $\ez > 0$ and $t \in (0,1]$, there exist $\delta_{0} \in (0,\tfrac{1}{2}]$ such that the following holds for all $\delta \in (0,\delta_{0}]$. Let $F \subset C^2([0,1]^{n-1},[0,1])$ be a non-empty $\dz$-separated $(\delta,t,\dz^{-\ez})$-set of cinematic family.
Then,
\begin{equation}\label{l2}
  \int_{[0,1]^{n} }  (\sum_{f \in F} \chi_{f^\dz}(x))^2 \, dx \lesssim \dz^{-2\ez}|F|^2\dz^{1+t}
\end{equation}
where $\chi_A$ is the characteristic function of a set $A \subset \R^n$.
\end{thm}

When the graphs of functions in $F$ in Lemma \ref{dznbhd1} and Theorem \ref{maxfcn} are $n-1$-spheres, inequalities \eqref{meas} and \eqref{l2} were known from the work by Kolasa-Wolff \cite[Lemma 2.1 and Theorem 1']{MR1722768}.
The proof of Theorem \ref{thmconfig} is also standard, which can date back to Cordoba \cite{Crdoba1977THEKM} who utilised the functional inequality to show that the planar Kakeya sets always have dimension $2$.


\subsection{Sketch the proof of Lemma \ref{dznbhd1}}
Finally, we briefly sketch the proof of Lemma \ref{dznbhd1}, which forms the core and the most technical part of the paper.
We will provide the heuristic idea
under the additional assumption $\|f-g\|_{C^2([0,1]^{n-1})} \sim 1$ for simplicity.

For some technical reason, we will decompose $[0,1]^{n-1}$ to dyadic cubes $U$ with diameter small enough only depending on  the cinematic constant $K$ as well as the modulus of continuity $\az$ and upper bound $\calL^{n}[f^\delta(U) \cap g^\delta(U)]$ for all $U$.
We fix $U \subset [0,1]^{n-1}$ and write $f^\dz \cap g^\dz = f^\delta(U) \cap g^\delta(U)$ for short. Noticing that the vertical thickness of $f^\dz \cap g^\dz$ is no more than $2\dz$, letting $P_{f,g}$ be the orthogonal projection of $f^\dz \cap g^\dz $ into the first $n-1$ coordinates, it suffices to show
\begin{equation}\label{sketch2}
   \calL^{n-1}(P_{f,g}) \lesssim \dz/\|f-g\|_{C^2([0,1]^{n-1})} \sim \dz.
\end{equation}

The proof consists two cases. Define the ``tangency parameter''
\begin{equation}\label{dz1}
  \Delta(f,g) := \inf_{x \in U} \{ |f(x)-g(x)| + |\nabla f(x)- \nabla g(x)|   \}.
\end{equation}

The first case is called the transversal case, where $\Delta(f,g) \sim \|f-g\|_{C^2([0,1]^{n-1})} \sim 1$, and roughly speaking, the intersection $f^\dz \cap g^\dz \subset U \times [0,1]$ looks like the one of two $c \times \cdots \times c\times \dz$ slabs in $U \times [0,1]$ where $c \sim \diam U$.
Thus we would expect $ \calL^{n-1}(P_{f,g}) \lesssim \dz$, i.e. $\calL^{n}(f^\dz \cap g^\dz) \lesssim \dz^2$ since $\sim \dz^2$ is the measure of the intersection of two transverse slabs in $\R^n$.
Actually, we will show this by foliating $P_{f,g}$ by an $(n-2)$-dimensional family of  curves $\{\gz_x: [0,b_x] \to U\}_{x \in \Lambda}$ where $\Lambda \subset \pa U$ and $\gz_x$ solves the ODE
\begin{equation}\label{sketch1}
  \dot \gz_x (t) = \frac{\nabla h(\gz_x (t))}{|\nabla h(\gz_x (t))|} \quad x \in \Lambda.
\end{equation}
Then we define a family of functions $h_x:= h \circ \gz_x$ of one variable. For these functions $h_x$, by employing the information that $\Delta(f,g) \sim \|f-g\|_{C^2([0,1]^{n-1})} \sim 1$ (applying an auxiliary Lemma \ref{pyz1} (1) in the real proof), we can show that
$$  |\{s \in [0,b_x]: |h_x(s)| \le 2\dz \}| \lesssim \dz. $$
Noting that
$$ P_{f,g}= \{x\in U :  |h(x)| \le 2\dz \} = \bigcup_{x\in \Lambda}  \{\gz_x(s): |h_x(s)| \le 2\dz \},$$
we will deduce \eqref{sketch2} if we can show the map $\Psi: \Omega:= \cup_{x \in \Lambda} (x,[0,b_x]) \to U$,
$$ (x,s) \mapsto \Psi(x,s):= \gz_x(s)$$
is a Lipschitz map with Lipschitz constant $\lesssim 1$. Indeed, since $\gz_x$ is the solution of \eqref{sketch1}, by the continuity dependence of the initial value in ODE theory, together with the $C^2$ regularity of $h$, we could obtain the desired quantitative bound $\| \mbox{Lip } \Psi\|_{L^\fz(\Omega)}\lesssim 1$.

\medskip
The second case is called the tangent case where the intersection $f^\dz \cap g^\dz$ looks like the one of two $\dz$-neighbourhoods of $(n-1)$-spheres ``$\lz$-almost tangent at one point'' in $\R^n$, meaning that the tangency parameter
$\Delta(f,g) \sim \lz \in [0, 1]$ is much smaller compared with the quantity $\|f-g\|_{C^2([0,1]^{n-1})}$.

Roughly speaking, when $\Delta(f,g)=0$, then there exists some point $x \in U$ such that the graphs of $f$ and $g$ share the same $(n-1)$-dim tangent plane at $x$.

To see $\calL^{n-1}(P_{f,g})\lesssim \dz$, we  give heuristic ideas in the special case that the graphs of $f$ and $g$ are (part of) upper-hemispheres and the tangency parameter $\lz \ge \dz$ (the case $\lz \le \dz$ is easier).
We start with two $2$-dimensional spheres $S^2$ in $\mathbb{R}^{3}$.

For convenience, we focus on the case that two centers of the (upper hemi)-spheres lie on the $x_3$-axis in $\mathbb{R}^3=\{(x_1,x_2,x_3)\}$ and they are the graphs $\Gamma_f$ and $\Gamma_g$ of $f$ and $g$, and the condition $\lz \ge \dz$ indicates that they are nearly inner tangent (but not perfectly tangent). See Figure \ref{sfig1} for an illustration.

\begin{figure}[h!]
\begin{center}
\begin{overpic}[scale = 0.3]{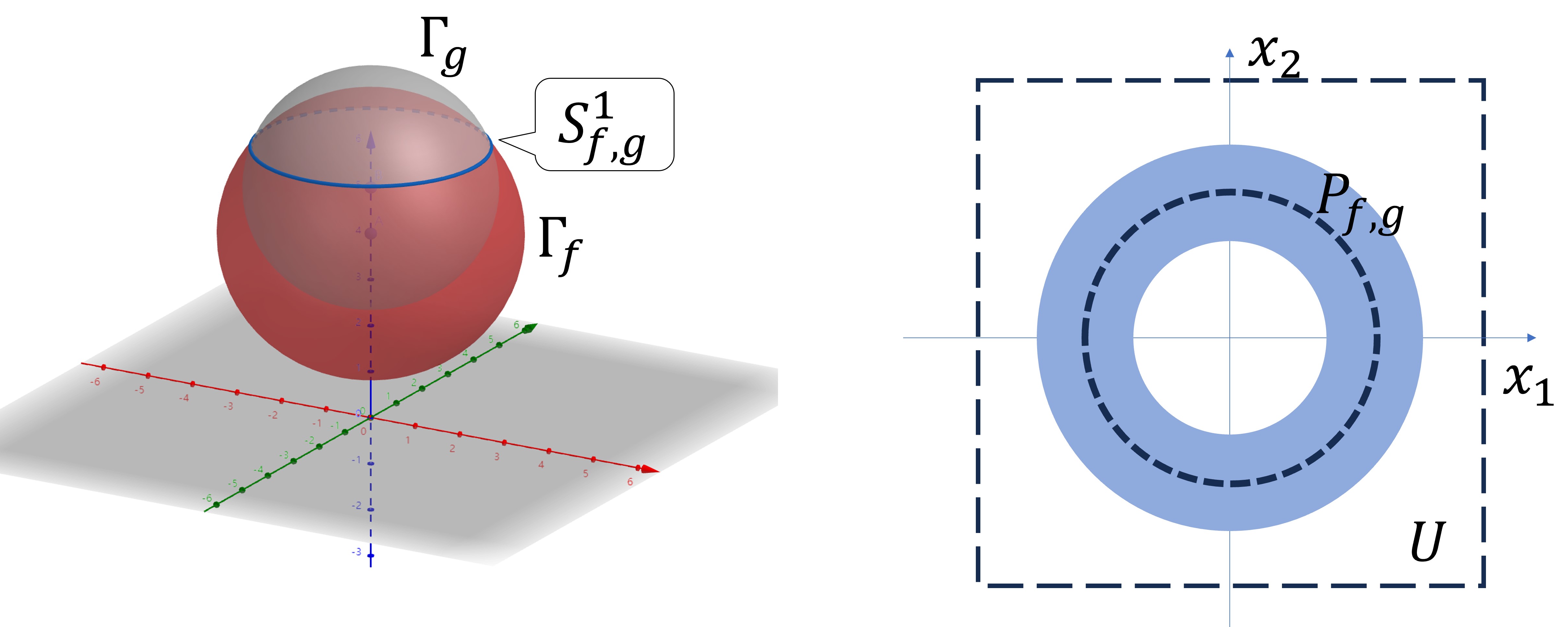}
\end{overpic}
\caption{The intersection of two spheres}\label{sfig1}
\end{center}
\end{figure}

The intersection of these two spheres $\Gamma_f \cap \Gamma_g$ is topologically a circle $S^1_{f,g}$ parallel to $x_1ox_2$-plane.
Thus $f^\delta \cap g^\delta$ is some "neighbourhood" of this $S^1_{f,g}$ in $\mathbb{R}^3$ which looks like a solid torus in $\mathbb{R}^3$. Moreover, $P_{f,g}$ is some neighbourhood of a circle in $x_1ox_2$-plane with radius $r$ same as the one of $S^1_{f,g}$. To estimate $\calL^{2}(P_{f,g})$, we just need to estimate two values: the radius $r$ of $S^1_{f,g}$ and the ``thickness'' of $P_{f,g}$.

To this end,
we decompose this 2-dimensional set $P_{f,g}$ to many $1$-dimensional intervals. For any unit vector $\xi$ in $x_1ox_2$-plane (i.e. $\xi \in S^1$), define
$$E_\xi:=\{s \ge 0 : s\xi \in U\} \mbox{ and } P_\xi:=\{s\xi : s   \in E_\xi\}.$$
Thus $U= \cup_{\xi \in S^1} P_\xi$ and $P_{f,g}= \cup_{\xi \in S^1} [P_\xi \cap P_{f,g}]$. The sets $\{P_\xi \cap P_{f,g}\}$ are the desired intervals (where the length $|\{P_\xi \cap P_{f,g}\}|$ corresponds to the ``thickness'' of $P_{f,g}$) and we will show
\begin{equation}\label{circle}
  \max_{x \in  \{P_\xi \cap P_{f,g}\}} |x| \lesssim \sqrt{\lz} \mbox{ and }  |P_\xi \cap P_{f,g}| \lesssim \sqrt{\dz/\lz} \quad \xi \in S^1.
\end{equation}
Geometrically, $P_\xi \cap P_{f,g}$ is the projection of the intersection of the (part of) plane $H_\xi=P_\xi \times$ span $\{(0,0,1)\}$ and the set $f^\dz \cap g^\dz$.
Observe that
$$ H_\xi \cap f^\delta \cap g^\delta$$
is the intersection of two $\delta$-neighbourhood of (part of) circles $S^1$ in $H_\xi$. We have transferred the estimate of $\calL^{2}(P_{f,g})$ to that of two $\delta$-neighbourhood of circles. See Figure \ref{sfig2} for an illustration.
\begin{figure}[h!]
\begin{center}
\begin{overpic}[scale = 0.3]{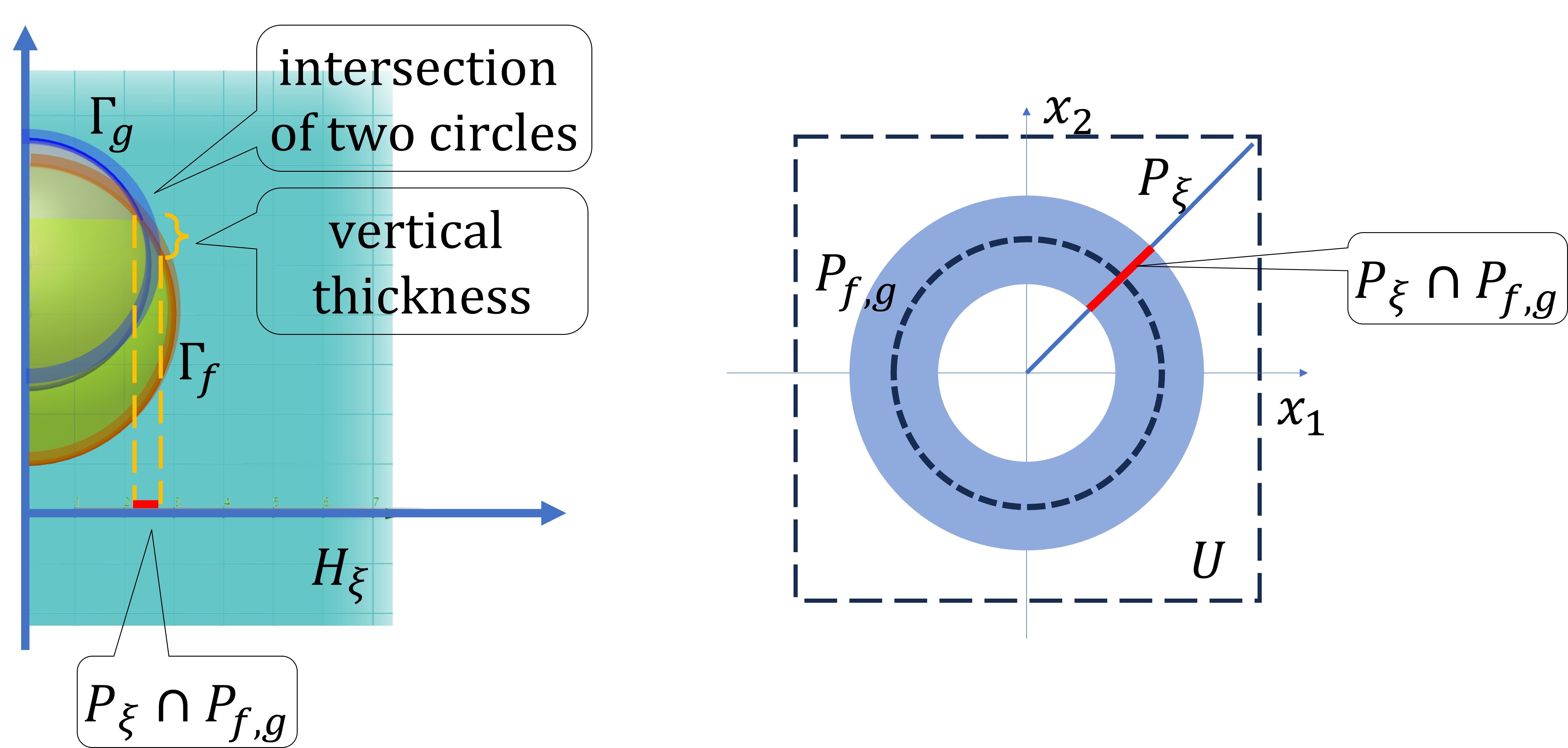}
\end{overpic}
\caption{The intersection of $H_\xi$ and two spheres}
\label{sfig2}
\end{center}
\end{figure}

For the $\delta$-neighbourhood of circles, the estimate \eqref{circle} has already been established by Kolasa-Wolff \cite[Lemma 3.3]{MR1722768}.
Consequently, noting that $(\xi,s)$ forms a natural polar coordinate, we use this polar coordinate to compute
\begin{align*}
    \calL^{2}(P_{f,g})& = \calL^{2}[\bigcup_{\xi \in \{S^{1}\}} |P_\xi \cap P_{f,g}| ] \lesssim \int_{\xi \in S^{1}} \int_{0}^{\sqrt{\lambda} } \chi_{P_\xi \cap P_{f,g}}(\xi s) s \, ds \, d\xi  \\
     & \lesssim  \int_{0}^{\sqrt{\lambda} } |P_\xi \cap P_{f,g}| \sqrt{\lambda} \, ds \\
     & \lesssim \sqrt{\lambda } \frac{\delta}{\sqrt{\lambda}} \le \delta.
  \end{align*}

  For higher $(n-1)$-dimensional spheres $S^{n-1}$ in $\mathbb{R}^{n}$, in this case the intersection of $\Gamma_f$ and $\Gamma_g$ will be some $(n-2)$-dimensional sphere $S^{n-2}_{f,g}$, so $f^\delta \cap g^\delta$ is some neighbourhood of this sphere $S^{n-2}_{f,g}$ in $\mathbb{R}^{n}$ with vertical thickness $\sim \delta$. Also, $P_{f,g}$ will be the neighbourhood of the projection of $S^{n-2}_{f,g}$ to span$\{x_1,x_2,\cdots x_{n-1}\}$.

  To compute $\calL^{n-1}(P_{f,g})$, we still use decomposition.
  Let $\xi \in S^{n-2} \subset $ span$\{x_1,x_2,\cdots x_{n-1}\}$ and $E_\xi, P_\xi, H_\xi$ be defined similarly as the above case.
  Then $P_\xi \cap f^\delta \cap g^\delta$ is still
the intersection of two $\delta$-neighbourhood of two circles $S^1$ in $H_\xi$ and thus the estimate \eqref{circle} holds.
Similarly we compute
\begin{align*}
    \calL^{n-1}(P_{f,g})& = \calL^{n-1}[\bigcup_{\xi \in \{S^{n-2}\}} |P_\xi \cap P_{f,g}| ] \lesssim \int_{\xi \in S^{n-2}} \int_{0}^{\sqrt{\lambda} } \chi_{P_\xi \cap P_{f,g}}(\xi s) s^{n-2} \, ds \, d\xi  \\
     & \lesssim  \int_{0}^{\sqrt{\lambda} } |P_\xi \cap P_{f,g}| \sqrt{\lambda}^{n-2} \, ds \\
     & \lesssim \sqrt{\lambda }^{n-2} \frac{\delta}{\sqrt{\lambda}} = \sqrt{\lambda }^{n-3}\delta \le \delta.
  \end{align*}
  Thus we get the desired result.
  From this computation, we also notice that the assumption $n \ge 3$ is necessary.

  \medskip
  In general, the graphs of $f$ and $g$ may not be (part of) $(n-1)$-dimensional spheres. Motivated by the above special case, we will still decompose $f^\dz \cap g^\dz$ into an $(n-2)$-dimensional family of slices $H_\xi \cap f^\dz \cap g^\dz$ such that $H_\xi \cap f^\dz \cap g^\dz$ looks like the intersection of two $\dz$-neighbourhoods of graphs of functions in a cinematic family of one variable for all $\xi \in S^{n-2}$. This decomposition is  guaranteed by the fact that $f-g$ is strictly convex/concave in $U$ which is further derived from our Definition \ref{cine} of the cinematic family.
  Thus
  for each intersection $P_\xi \cap P_{f,g}$, we will show that the estimate \eqref{circle} holds for a general cinematic family with the help of the auxiliary Lemma \ref{pyz1} which is a slight modification of \cite[Lemma 3.8]{2022arXiv220702259P}.
  Finally, with the estimate \eqref{circle}, similar to the above computation, we use the integration under the polar coordinate to deduce Lemma \ref{dznbhd1} as desired.

\bigskip
\noindent
{\bf Organization.}
In Section 2, we recall several definitions and give preliminary properties of cinematic families. In Section 3, we prove Lemma \ref{dznbhd1} and Theorem \ref{maxfcn}. In Section 4, we show Theorem \ref{thmconfig} and Theorem \ref{dimf}. Finally, in Section 5, we apply Theorem \ref{thmconfig} to deduce Theorem \ref{mainproj2} and Theorem \ref{mainproj}. Also, we give the new proof for the Mattila's projection theorem in $\R^n$ with $n \ge 3$.

\bigskip
\noindent
{\bf Notation.}
The notation $A \lesssim B$ means that there exists a constant $C \geq 1$ such that $A \leq CB$. The two-sided inequality $A \lesssim B \lesssim A$ is abbreviated to $A \sim B$. The constant $C$ can depend on the dimension $n$, and if we fix a cinematic family $F$ of cinematic constant $K$ and doubling constant $D$, $C$ may depend on $K$ and $D$.
If we want to stress the dependence of the constant $C$ with some  parameter "$\theta$", we indicate this by writing $A \lesssim_{\theta} B$.

By the notation $|A|$ we denote the cardinality of $A$ and by $\calH^{s}(A)$ the $s$-dimensional Hausdorff measure of $A$.

For two vectors $u,v \in \mathbb{R}^n$,
$u \cdot v$ and $\langle u ,v \rangle$ both denote the Euclidean inner product. Given $f \in C^2(\Omega)$ for some $k$-dim $C^2$ manifold $\Omega \subset \mathbb{R}^n$,
$\nabla f$ and $\nabla^2 f$ denote the gradient and Hessian of $f$ respectively. Moreover, as usual $\nabla_u f:= \langle \nabla f ,u \rangle$ and $\nabla_v\nabla_u f:=  \langle \nabla^2 f \cdot u ,v \rangle$.
If $\Omega \subset \mathbb{R}$, we write $f'=\nabla f $ and $f''=\nabla^2 f$. As usual, $T_x\Omega=\R^k$ is the tangent space of $x \in \Omega$ and for $\xi \in T_x\Omega$ with $|\xi|=1$, we always write $\xi \in S^{k-1}$ for simplicity.

\bigskip
\noindent
{\bf Acknowledgement.}
The author would like to thank Katrin F\"{a}ssler and Tuomas Orponen for many motivating discussions and helpful comments.

\section{Preliminaries}

The next lemma is a multivariable variant of \cite[Lemma 3.1]{marstrandtype} and \cite[Lemma 3.5]{2022arXiv220702259P}.
\begin{lemma} \label{local}
Assume $\eta>0$, $n \ge 1$.
Let
$F \subset C^{2}([0,1]^{n-1},[0,1])$ be a cinematic family with cinematic constant $K$ and modulus of continuity $\az$ and $f,g \in F$ with $t:=\|f-g\|_{C^2( I^{n-1}) } \le 1$ where $I^{n-1} \subset \R^{n-1}$ is a dyadic cube.
  If $U \subset I^{n-1}$ is a (closed) dyadic cube such that its diameter is smaller than $\az(\eta/2)$, and $k$ is one of the three functions $|f-g|, |\nabla (f-g)|, |\nabla_\xi \nabla_\xi (f-g)|$ for some $\xi \in S^{n-2}$, then $k$ satisfies either one of the following two inequalities:
   \begin{itemize}
     \item[(S)] $ k(x) < \eta\|f-g\|_{C^2( I^{n-1}) } \mbox{ for all } x \in U$.
     \item[(L)] $k(x) \ge \eta \|f-g\|_{C^2( I^{n-1})}/2 \mbox{ for all } x \in U$.
   \end{itemize}
\end{lemma}

\begin{proof}
 Since $\diam U < \az(\eta/2)$, for any $x,x' \in U$, $|x-x'| < \az(\eta/2)$.
 By Definition \ref{cine}, we know
 that $|\nabla_\xi \nabla_\xi (f-g)(x)-\nabla_\xi \nabla_\xi (f-g)(x')| <\eta/2$.
 Moreover, since $F$ is contained in a ball of diameter $K$, we know
 $\|f-g\|_{C^{2}([0,1]^{n-1})} \le K$. In particular, we know Lip$(f-g)(x) \le K$ and Lip$|\nabla(f-g)|(x)\le K$ for all $x \in U$. Recalling from Definition \ref{cine} that $\az(s) \le K^{-1}s$, we know
 $$  |k(x)-k(x')| \le \frac{\eta}2, \quad \forall x,x' \in U,  $$
 where $k$ is one of the three functions $|f-g|, |\nabla (f-g)|, |\nabla_\xi \nabla_\xi (f-g)|$.

 Next, fix $U \subset I^{n-1}$ with $\diam U < \eta/2$, and suppose that the
alternative (L) fails for $k = f - g$. So there is $x_0 \in U$ such that $|k(x_0)| < \eta \|f-g\|_{C^2( I^{n-1})}/2$. And for any $x \in U$, since $|x-x_0| < \az(\eta/2)$, by defining $l= k/\|f-g\|_{C^2(I^{n-1})}= k/t$, we deduce that Lip$_{I^{n-1}}l \le 1$ and
$$ \frac{|k(x)|}{t} \le \frac{|k(x_0)|}{t}+ \left| \frac{|k(x)|}{t} - \frac{|k(x_0)|}{t}\right| < \frac{\eta}{2} + |l(x)-l(x_0)| \le \frac{\eta}{2} +\mbox{ [Lip$_{I^{n-1}}l$}] |x-x_0|  \le \eta.$$
 Thus the alternative (S) holds for $k$. The proof of the cases $k = |\nabla (f-g)|$ and $k= |\nabla_\xi \nabla_\xi (f-g)|$ are the
same.
\end{proof}

We denote by $\lz U$ the cube with the same center as $U$ and has side length $\lz$ times as that of $U$ for $\lz>0$.
In the following proof, for technical reasons, given a cinematic family $F \subset C^{2}([0,1]^{n-1},[0,1])$ with cinematic constant $K$, we first extend $F$ to be a cinematic family in $C^{2}([-\frac{\az(K^{-1}/6)}2,1+\frac{\az(K^{-1}/6)}2]^{n-1})$. Thus $F$ will have a new cinematic constant $\bar K \ge K$ and a new modulus of continuity $\bar \az \le \az$.
Since $\bar\az,\az$ are increasing functions, $\bar\az(\bar K^{-1}/6) \le \az(\bar K^{-1}/6) \le \az(K^{-1}/6)$. Thus $F$ is also a cinematic family in $C^{2}([-\frac{\bar\az(\bar K^{-1}/6)}2,1+\frac{\bar\az(\bar K^{-1}/6)}2]^{n-1})$ with cinematic constant $\bar K$ and modulus of continuity $\bar \az$.
By abuse of notation, we still write $\bar K=K$ and $\bar \az =\az$. As a result, letting $I_F^{n-1}:=[-\frac{\az(K^{-1}/6)}2, 1+\frac{\az(K^{-1}/6)}2]^{n-1}$, we know $F \subset C^{2}(I_F^{n-1})$ is a cinematic family of cinematic constant $K$ and modulus of continuity $\az$.

%
The above lemma gives the following corollary.

\begin{cor} \label{uniq}
  Let
$F \subset C^{2}(I_F^{n-1})$ be a cinematic family with cinematic constant $K$ and modulus of continuity $\az$, $f,g \in F$ and $h:=f-g$. Then for each dyadic cube $U \subset [0,1]^{n-1} \subset I_F^{n-1}$ with $\diam U < \frac{\az(K^{-1}/6)}2$,

  {\rm(A):} at least one of the following three conditions hold: \begin{enumerate}
  \item[(i)]
$|h(x)| \ge \frac{1}{3K} \|h\|_{C^2(I_F^{n-1}) }$ for all $x \in  2U$.
  \item[(ii)]
$|\nabla h(x)| \ge \frac{1}{3K} \|h\|_{C^2( I_F^{n-1}) }$ for all $x \in  2U$.
  \item[(iii)]
   \begin{equation}\label{L}
    |\nabla_\xi \nabla_\xi h (x)| \ge \frac{1}{3K} \|h\|_{C^2( I_F^{n-1}) } \qquad \mbox{for all } x \in  2U, \xi \in S^{n-2}.
  \end{equation}
\end{enumerate}

  {\rm(B):} Moreover, if $|\nabla_\xi \nabla_\xi h|$ satisfies \eqref{L}, then there exists at most one point $x \in 2U$ such that $\nabla h(x)=0$ and $h$ is strictly convex/concave in $2U$.
\end{cor}

\begin{proof}
   To see (A),  we apply Lemma \ref{local} to $h$ with $I^{n-1}=I_F^{n-1}$ and $\eta = K^{-1}/3$. As a result, we know for each dyadic cube $U \subset [0,1]^{n-1}$ with diameter $\diam U < \frac{\az(\eta/2)}2 = \frac{\az(K^{-1}/6)}2$ (thus $\diam (2U) < \az(\eta/2)$),
 either (S) or (L) holds for $k$ and all $x \in 2U$($\subset I_F^{n-1}$), where $k$ is one of the
functions $h$, $\nabla h$ or $\nabla_\xi \nabla_\xi h$, $\xi \in S^{n-2}$. On the other hand,
since $F$ is a cinematic family of functions, if the alternative (L) does not hold for the functions $h$ and $\nabla h$, then by Definition \ref{cine}(3), for all $\xi \in S^{n-2}$, (L) holds for $\nabla_\xi \nabla_\xi (f-g)$.

   To see the assertion (B), assume on the contrary that there exists two points $x,y \in 2U$ such that $\nabla h(x)=\nabla h(y)=0$.
   In particular, letting $e=\frac{x-y}{|x-y|}$, we have $\nabla_e h(x)=\nabla_e h(y)=0$.
   Let $\gz:[0,|x-y|] \to 2U$, $\gz(t)=(1-\frac{t}{|x-y|})x + \frac{t}{|x-y|} y$. Then $h\circ \gz  \in C^2([0,|x-y|],\R)$, $(h\circ \gz)'(t) = \nabla h(\gz(t)) \cdot \dot \gz(t) = \nabla_e h(\gz(t))$ and $(h\circ \gz)''(t)  = \nabla_e\nabla_e h(\gz(t))$. By Taylor's formula, there exists $t_0 \in [0,|x-y|]$ such that
   $$ 0= \nabla_e h(x)-\nabla_e h(y) = |x-y| \nabla_e\nabla_e h(\gz(t_0)).  $$
   Hence $\nabla_e\nabla_e h(\gz(t_0))=0$, which contradicts \eqref{L}.

   To see $h$ is strictly convex/concave in $2U$, recalling $h \in C^2(2U)$, we know $\nabla_\xi \nabla_\xi h (x)$ is continuous with respect to $x$ and $\xi$. Thus the assumption \eqref{L} holds implies that $\nabla_\xi \nabla_\xi h (x) >\frac{1}{3K} \|h|_{\lz U}\|_{C^2( 2 U) } >0 $ or $\nabla_\xi \nabla_\xi h (x) < -\frac{1}{3K} \|h|_{\lz U}\|_{C^2( 2 U) } <0 $ for all $x \in 2U$ and $\xi \in S^{n-2}$.
   If the first case happens, for all $\xi \in S^{n-2}$ and $t$ sufficiently small,
we have
$$  h(x+t\xi)  = h(x) +  t \nabla_\xi h(x) + t^2 \nabla_\xi \nabla_\xi h(x) + o(t^2) > h(x) +  t \nabla_\xi h(x).   $$
Thus $h$ is strictly convex at all $x \in 2U$. Similarly, the second case implies that $h$ is concave. The proof is complete.
\end{proof}

%
%

We recall the following facts from ODE.
Consider the following dynamical system in a domain $\Omega \subset \R^{n-1}$:
\begin{equation}\label{dyn1}
  \dot x = f(x)  \quad x \in \Omega
\end{equation}
where $f \in C^1(\Omega, \R^{n-1})$.
Define the flow of \eqref{dyn1} by $\Phi(x,s)$, i.e. for each $x \in \Omega$, $\pa_s \Phi(x,s) = f(\Phi(x,s))$ and $\Phi(x,0)= x$.
Write $\nabla_x \Phi(x,s)= (\pa_{x_1}\Phi(x,s), \cdots, \pa_{x_{n-1}}\Phi(x,s))$. By ODE theory (e.g. \cite[Theorem in page 154]{MR1629775}), we know $\Phi$ is $C^1$ and
$$ |\nabla_{x_i} \Phi(x,s)| \le \|f\|_{L^\fz(\Omega)} + \|\nabla f\|_{L^\fz(\Omega)} \int_{s_0}^s |\nabla_{x_i} \Phi(x,\tau)| \, d\tau \quad i=1,\cdots,n-1 $$
for any $s,s_0$ such that the flow $\Phi$ is defined.
Using the Gronwall's inequality one has
\begin{equation}\label{flowlip}
  |\nabla_x \Phi(x,s)| = |(\pa_{x_1}\Phi(x,s), \cdots, \pa_{x_{n-1}}\Phi(x,s))| \lesssim \|f\|_{L^\fz(\Omega)} e^{\|\nabla f\|_{L^\fz(\Omega)}}  \lesssim \|f\|_{C^1(\Omega)}.
\end{equation}


We recall that $A^\circ$ is the interior of a set $A \subset \R^n$.

\begin{lemma} \label{extend}
  Let $h$ be as in Corollary \ref{uniq} and assume $\nabla h \ne 0$ in $2U$. Consider the system
  \begin{equation}\label{tan0}
     \dot \gz(s) = \nabla h ( \gz(s)) \quad   in \ 2U.
  \end{equation}
  Then each solution $\gz$ of \eqref{tan0} can be extended to the boundary of $2U$.
\end{lemma}

\begin{proof}
  Given any $x \in (2U)^\circ$, by the classical ODE theory, there exists a unique solution $\gz_x : (-\ez,\ez) \to 2U$ with $\gz_x(0)=x$. We show $\gz_x$ can be extended to a curve (still denote by $\gz_x$) defined on some interval $[a_x,b_x]$ such that $\gz_x(a_x), \gz_x(b_x) \in \pa (2U)$.

  Indeed, recalling that $\|h\|_{C^2(2U)} \le K$, we know $|h| \le K$ in $2U$. On the other hand, by assumption, $\nabla h \ne 0$ in  $2U$ implies that $\min_{x \in 2U} |\nabla h(x)| \ge a$ for some $a >0$. Thus,
  $$ 2K \ge h(\gz_x(t)) - h(x) = \int_0^t \nabla h (\gz_x(s)) \cdot \dot \gz_x(s)\, ds  = \int_0^t |\nabla h (\gz_x(s))|^2 \, ds \ge a^2 t.  $$
  This implies that $b_x \le 2K/ a^2$ and $\gz_x(b_x) \in \pa (2U)$. Similarly, we obtain that $\gz_x(a_x) \in \pa (2U)$.
\end{proof}

\begin{lemma} \label{tran3}
  Let $h$ be as in Corollary \ref{uniq} and the case (A)(ii) holds for $h$.
  Consider the system
  \begin{equation}\label{tran}
     \dot \gz(s) = \frac{\nabla h ( \gz(s))}{|\nabla h ( \gz(s))|}  \quad   in \ U.
  \end{equation}
  Set
  $$\pa U^+:= \{ x \in \pa U : \mbox{there exists a maximal solution $\gamma_x \in C^2( [0,b_x] ,U)$ to \eqref{tran} starting from $x$} \}.$$
  Then the family $\{\gz_x\}_{x \in \pa U^+}$ foliates $U$, that is,
  $$ U= \bigcup_{x \in \pa U^+}\gamma_{x}, \text{ and } \gamma_{x_1} \cap \gamma_{x_2}=\emptyset \text{ if $x_1 \ne x_2$.} $$
\end{lemma}

\begin{proof}
  Note that (A)(ii) in Corollary \ref{uniq} holds implies that $\nabla h \ne 0$ in $U$. Thus $\frac{\nabla h ( \gz(s))}{|\nabla h ( \gz(s))|}$ is well defined in $2U$. Moreover, up to a parametrization, the maximal solution $\gamma_x$ is the same as the one of the following system
  $$ \dot \gz(s) = \nabla h ( \gz(s))  \quad   in \ U. $$
  Thus for any $y \in U^\circ$, there exists a unique solution $\gz_y : (-\ez,\ez) \to U$ with $\gz_y(0)=y$. By Lemma \ref{extend}, we know $\gz_y$ extends to a maximal solution defined on $[a_y,b_y]$ (still denote by $\gz_y$) such that $\gz_y(a_y) \in \pa U$. This implies $\gz_y(a_y) \in \pa U^+$. Let $x=\gz_y(a_y)$. We know $y \in \gz_x$. Thus, $U = \cup \{\gz_x\}_{x \in \pa U^+}$.
\end{proof}

\begin{remark} \label{tran2}\rm Under the assumption as in Lemma \ref{tran3},
  note that the dyadic cube $U \subset \R^{n-1}$ has $2^{n-1}$ facets $\pa U_1, \cdots, \pa U_{2^{n-1}}$. Let
  $$ \pa U^+_i :=  \pa U_1 \cap \pa U^+ \mbox{ and } U_i:= \bigcup_{x \in \pa U^+_i} \gz_x \quad i=1,\cdots,2^{n-1}. $$
  For each $\pa U_i$, by possibly reordering and translating the coordinates, we can assume $\pa U_i \subset$ span $\{x_1,\cdots,x_{n-2}\}$. For each $x=(x_1,\cdots,x_{n-1}) \in \pa U_i$, write $\bar x=(x_1,\cdots,x_{n-2})$. Thus we can equip $U_i$ with a new coordinate by letting $y=(\bar x,s)$ for all $y \in U_i$ where $x \in \pa U_i^+$ and $s \in [0,b_x]$ is unique pair such that $y = \gz_x(s)$.
\end{remark}

The following lemma is a slight variant of
\cite[Lemma 3.8]{2022arXiv220702259P} and the proof is essentially the same as \cite[Lemma 3.8]{2022arXiv220702259P}.

\begin{lemma}  \label{pyz1}
  Let $[0,a] \subset \R$ and $h \in C^2([0,a])$ with $\|h\|_{C^2([0,a])}=t$.
  For $\dz>0$,
  write
  $$  E_{2\dz}=\{s \in [0,a] : |h(s)| \le 2\dz\}.  $$

  Then, for each $0<\dz < c_1 t$ for some constant $0<c_1<1$ and $4c_1 \le c_2<1$,
  \begin{enumerate}
    \item if
    $$\lz_1:= \inf_{s \in [0,a]}\{|h(s)| + |h'(s)|\} \ge c_2 t,$$ then $|E_{2\dz}| \lesssim \dz/t$;
    \item assume $\lim_{s \to 0+}h'(s)=0$ and
    $$ h''(s) \ge c_2 t, \quad s \in (0,a).   $$
    Let
        $$\lz_2:= |h(0)|, $$
         then $E_\dz \subset [0,d]$ with $d \le c_2^{-1}\sqrt{(\lz_2+\dz)/ t}$. Moreover, $E_\dz$ is a closed interval of length $\lesssim \dz/\sqrt{(\lz_2+\dz)t}$.
  \end{enumerate}

\end{lemma}

\begin{proof}
  First, we show (1). Noticing that for $s \in E_{2\dz}$, we have
  $$  c_2t \le \lz_1 \le \min_{s \in E_{2\dz}}\{|h(s)| + |h'(s)|\} \le 2\dz + \min_{s \in E_{2\dz}}\{|h'(s)|\}. $$
  By $4c_1 \le c_2$, we know $\min_{s \in E_{2\dz}}\{|h'(s)|\} \gtrsim t$. Thus we get $|E_{2\dz}| \lesssim \dz/t$ as desired.

  Next, we show (2). Using  $h''(s) \ge c_2 t$, we deduce that for all $s \in [0,a]$, $h'(s)$ does not change sign and
  \begin{align*}
    |h(s)| & \ge |\int_{0}^s  h'(\tau) \, d \tau| -\lz_2 =  \int_{0}^s  |h'(\tau)| \, d \tau - \lz_2 \\
     & = \int_{0}^s  \int_{0}^\tau  h''(\rho) \, d \rho \, d \tau  - \lz_2 \ge c_2 t s^2 - \lz_2.
  \end{align*}
  Noticing that $s \in E_{2\dz}$ implies $|h(s)| \le 2\dz$, we know
  $\max E_{2\dz} \le  \sqrt{(\lz_2+\dz)/(c_2t)} \le c_2^{-1}\sqrt{(\lz_2+\dz)/t}$. 

  Finally, we estimate $|E_{2\dz}|$. In the case $\lz_2 \le \dz$, it holds
  $$ |E_{2\dz}| \le d \lesssim \sqrt{(\lz_2+\dz)/t} \sim \frac{\lz_2+\dz}{\sqrt{(\lz_2+\dz)/t}} \sim \frac{\dz}{\sqrt{(\lz_2+\dz)/t}}. $$
  In  the case $\lz_2 > \dz$, recalling $\lz_2 =|h(0)|$, for each $s \in E_{2\dz} \cap [0,a]$, we deduce that
  \begin{align*}
    2\dz & \ge |h(s)| \ge \lz_2-\int_{0}^s  |h'(\tau)| \, d \tau  \\
     & \ge \lz_2- \int_{0}^s  \int_{0}^\tau  |h''(\rho)| \, d \rho \, d \tau   \ge \lz_2 -  t s^2 ,
  \end{align*}
  which implies $s \gtrsim \sqrt{(\lz_2 + \dz)/t} \sim \sqrt{\lz_2 /t} $ for all $E_{2\dz} \cap [0,b]$. 
  Using $h''(s) \ge c_2 t$ for all $s \in (0,a)$, we have
  $$ |h'(s)|= |h'(s)- h'(0)| \gtrsim t|s| \ge \sqrt{\lz_2 /t}, \quad s \in E_{2\dz} .  $$
  As a result, we get $|E_{2\dz}| \lesssim \dz/\sqrt{\lz_2 /t} \sim \dz/\sqrt{(\lz_2 + \dz)/t}$ as desired.
\end{proof}

We end this section by recalling the following fact for convex sets (see for example \cite[Theorem 3.6 (b6)]{MR4180684}).

\begin{lemma}\label{convset}
  Let $A \subset B \subset \R^{n}$ be two compact convex sets with $\calL^n(A)>0$. Then
  $$  \calH^{n-1}(\pa A) \le \calH^{n-1}( \pa B). $$
\end{lemma}



\section{Intersections of functions in a cinematic family}
We first show Lemma \ref{dznbhd1} in this section. 
Recall that for arbitrary $\eta>0$ and
subdomain $D \subset I^{n-1}$, the vertical $\eta$-neighborhood is
   $$f^{\eta}(D):= \left\{(x,y)=(x_1,\cdots,x_{n-1},y) \in D\times\mathbb{R} : f(x)- \eta \le y \le f(x)+\eta\right \}.$$




%
%
%
%
%
%
%

\begin{proof} [Proof of Lemma \ref{dznbhd1}]
Write $t:=\|f-g\|_{C^2([0,1]^{n-1})}$.
Similar to the previous section, we may assume $F \subset C^{2}(I_F^{n-1})$ be a cinematic family with cinematic constant $K$ and modulus of continuity $\az$.
By
decomposing $[0,1]^{n-1}$ by dyadic cubes $\{U\}$ with $\frac{\az(K^{-1}/6)}4 < \diam U < \frac{\az(K^{-1}/6)}2$, to prove \eqref{meas}, it suffices to show
$$ \calL^{n}(f^\delta(U) \cap g^\delta(U)) \lesssim \frac{\dz^2}{t}   $$
for each $U \subset[0,1]^{n-1}$.
Notice that for each $U \subset [0,1]^{n-1}$, $f,g$ are well-defined on $2U$.
In the following, we fix $U$ and write $f^\delta \cap g^\delta = f^\delta(U) \cap g^\delta(U)$ for simplicity.

  If $f^\delta \cap g^\delta = \emptyset$, then there is nothing to prove. Also, since $t \in [\delta,1]$, if $\delta \le t \le [\frac{60K}{\az(K^{-1}/6)}]^2\delta$,
   then $\calL^{n}(f^\delta \cap g^\delta) \le \calL^{n}(f^\delta) \sim \delta \sim \frac{\delta^2}{t}$.
Thus in the following, we assume
\begin{equation}\label{ass1}
  t > 4[\frac{10^5K^4}{\az(K^{-1}/6)^2}]^2\delta.
\end{equation}

  Assume $f^\delta \cap g^\delta \ne \emptyset$.
  Let $P_{f,g}$ be the orthogonal projection of
  $f^\delta \cap g^\delta$ to $U$.
Since the vertical thickness of $f^\delta \cap g^\delta$ is no more than $2\delta$, we need to show
  \begin{equation}\label{Pfg}
   \calL^{n-1} (P_{f,g}) \lesssim \frac{\delta}{t}.
  \end{equation}

  Write $h=f-g$. By $f^\delta \cap g^\delta \ne \emptyset$, we know there exists $x \in U$ such that $|h|\le 2\delta$ which implies $|h|$ does not enjoy Corollary \ref{uniq} (A)(i).
  Hence either Corollary \ref{uniq} (A)(ii) (the transversal case) or Corollary \ref{uniq} (A)(iii) (the tangent case) holds.
  We consider these two cases separately.

  \medskip
  \emph{Case 1.} Corollary \ref{uniq} (A)(ii) holds, that is
  $$  |\nabla h(x)| \ge \frac{1}{3K} \|h\|_{C^2(I_F^{n-1})} \quad  x \in 2U.  $$
In particular, $\nabla h \ne 0$ in $U$.

We will prove \eqref{Pfg} by showing that $P_{f,g}$ is the image of a set $\Omega$ under a Lipschitz map $\Psi$ with Lipschitz constant $\lesssim 1$ and $\calL^{n-1}(\Omega) \lesssim \delta/t$. To this end, we divide the proof into two steps. As mentioned in the introduction, in \emph{Step 1}, we foliate $P_{f,g}$ by a family of curves $\gz_x$ using Lemma \ref{tran3}. And we define a family of auxiliary functions $h_x$ derived from $h$ and $\gz_x$ which satisfies Lemma \ref{pyz} (1). In \emph{Step 2}, we apply Lemma \ref{pyz1} (1) to conclude \eqref{Pfg}.

\medskip \emph{Step 1}.
  Consider the ODE
  \begin{equation}\label{gflow}
    \dot \gamma(s) = \frac{\nabla h(\gamma)}{|\nabla h(\gamma)|} \quad  \mbox{in } U.
  \end{equation}
  Since $\nabla h \in C^1(U,\R^{n-1})$ and $\nabla h \ne 0$ in $U$, we know $\frac{\nabla h}{|\nabla h|} \in C^1(U,\R^{n-1})$. By Lemma \ref{tran3}, we know, for any $x \in U$, there exists a unique
   the maximal solution $\gamma_x \in C^2( [a,b] ,U)$ to \eqref{gflow} with $\gamma_x(0)=x$, and $ \gamma_x(a),\gamma_x(b) \in \pa U$.
   Denote
   $$\pa U^+:= \{ x \in \pa U : \mbox{there exists a solution $\gamma_x \in C^2( [0,b_x] ,U)$ to \eqref{gflow} starting from $x$} \}.$$
Then, again by Lemma \ref{tran3}, we know $\{\gamma_x\}_{\xi \in \pa U^+}$ gives a foliation of $U$.

  For each $x  \in \pa U^+$, define $h_x:= h\circ \gamma_x$. Then $h_x \in C^2([0,b_x],U)$. Our goal is to show $|\{ |h_x| \le 2\delta\}|  \sim \delta/ t$.
  To this end, we check that $h_x$ satisfies Lemma \ref{pyz1} (1).
  Define
  $$ \lz= \Delta(f,g) :=  \inf_{x \in U} \{ |h(x)| + | \nabla h(x)|\},  $$
  and
  $$  \lambda_x:= \inf_{s \in [0,b_x]} \{ |h_x(s)| + | h_x'(s)|\} \text{ and } t_x:=\|h_x\|_{C^2([0,b_x])}.  $$
  We show
  \begin{equation}\label{tlz}
    3K^{-1} t \le \lambda \le  \lambda_x \le t_x \le t  \quad  x \in \pa U^+.
  \end{equation}
  In particular,
   \begin{equation}\label{lzxtx}
     \lambda_x \ge \frac{1}{3K}t_x \quad x \in \pa U^+.
   \end{equation}
   To see \eqref{tlz}, we compute
   $$h_x'(s)= \nabla h(\gamma_{x}(s)) \cdot \dot \gamma_{x}(s) = \nabla h(\gamma_{x}(s))\cdot\frac{\nabla h(\gamma_{x}(s))}{|\nabla h(\gamma_{x}(s))|} = |\nabla h(\gamma_{x}(s))|.$$
   This implies
   $\lz_x \ge  \lz= \Delta(f,g)$. Also, it is immediate that $t_x \le t$. On the other hand, the fact that $|\nabla h|$ satisfies Corollary \ref{uniq} (A)(ii) implies
   $$\lambda \ge 3K^{-1} \|h\|_{C^2(I_F^{n-1})} \ge 3K^{-1} \|h\|_{C^2([0,1]^{n-1})} = 3K^{-1} t.$$
    Thus \eqref{tlz} and \eqref{lzxtx} hold.

%



   \medskip
   \emph{Step 2.} First, notice that \eqref{tlz} also gives $t_x \ge (3K)^{-1}t$. Combining \eqref{ass1}, we have $t_x \ge \frac{10^9K^7}{\az(K^{-1}/6)^4}\delta$ for all $x \in \pa U^+$.
   Applying Lemma \ref{pyz1} (1) to each $h_x$ with $c_1=\frac{\az(K^{-1}/6)^4}{10^9K^7} \le \frac{1}{12K}$ and $c_2=3K^{-1} \ge 4c_1$, we know
   \begin{equation}\label{1dim}
     |I_x:=\{s \in [0,b_x]: |h_x(s)| \le 2\delta\}| \lesssim \delta/\sqrt{\lambda_x t_x} \sim \delta/ t.
   \end{equation}
In the following, by Remark \ref{tran2}, we can estimate $\calL^{n-1}(P_{f,g})$ by estimating $\calL^{n-1}(P_{f,g} \cap U_i)$ for each $i=1,\cdots 2^{n-1}$. Thus we may without loss of generality assume that $P_{f,g} \subset U_i \subset$ span $\{x_1,\cdots,x_{n-2}\}$ for some $i$ and thus $\pa U^+ \subset \pa U_i$.
For any $x=(x_1,\cdots,x_{n-1}) \in \pa U^+ \subset 2U \subset \R^{n-1}$, write $x=(\bar x,x_{n-1})$ as in Remark \ref{tran2}.
   Let $\Omega:= \cup_{x \in \pa U^+} \{(\bar x, I_x)\} \subset \R^{n-1}$ and define the map $\Psi : \Omega \to U$ by
  $$ \Psi(\bar x,s) = \gz_x(s).   $$
  Since $2U \subset I_F^{n-1}$, the ordinary equation
  \begin{equation}\label{gflow5}
    \dot \gamma(s) = \frac{\nabla h(\gamma)}{|\nabla h(\gamma)|} \quad  \mbox{in } 2U
  \end{equation}
  is well-defined. Let $\Phi(x,s)$ be the flow of \eqref{gflow5} with initial condition $\Phi(x,0)=x$ for all $x \in 2U$. Then we know for all
  \begin{equation}\label{phi}
   \Psi(\bar x,s)=\Phi(x,s) \quad  (\bar x,s) \in \Omega.
  \end{equation}
  Note that
   $\frac{\nabla h}{|\nabla h|} \in C^1(2U,\R^{n-1})$ and
   $$ \nabla(\frac{\nabla h}{|\nabla h|})(x) = \frac{\nabla^2 h}{|\nabla h|}(x) -  \frac{(\nabla^2 h \cdot \nabla h) \otimes \nabla h}{|\nabla h|^3}(x)  \quad \forall x \in 2U . $$
   Recall the standing assumption that $|\nabla h(x)| \ge \frac{1}{3K} \|h\|_{C^2(I_F^{n-1})} \ge \frac{1}{3K} \|h\|_{C^2(2U)}$, we know
   \begin{equation}\label{hc1}
     \|\nabla(\frac{\nabla h}{|\nabla h|})\|_{L^\fz(2U)} \le 6K\frac{\|\nabla^2 h\|_{L^\fz(2U)} }{\|h\|_{C^2(2U)}} \lesssim 1 .
   \end{equation}
   Recalling \eqref{flowlip}, one has
   $$|\nabla_x \Phi(x,s)| = |(\pa_{x_1}\Phi(x,s), \cdots, \pa_{x_{n-1}}\Phi(x,s))| \lesssim e^{\|\nabla(\frac{\nabla h}{|\nabla h|})\|_{L^\fz(2U)}} \lesssim 1 \quad  (\bar x,s)\in \Omega. $$
    Then by \eqref{phi}, for $(\bar x,s)\in \Omega$, write $\nabla_{\bar x} \Psi(\bar x,s)=(\pa_{x_1}\Psi(\bar x,s), \cdots, \pa_{x_{n-2}}\Psi(\bar x,s))$. We have
   $$ |\nabla_{\bar x} \Psi(\bar x,s)| = |(\pa_{x_1}\Psi(\bar x,s), \cdots, \pa_{x_{n-2}}\Psi(\bar x,s))|= |(\pa_{x_1}\Phi(x,s), \cdots, \pa_{x_{n-2}}\Phi(x,s))|  \lesssim 1 . $$

   On the other hand, we have $|\pa_s \Psi(\bar x,s)| = |\dot \gz_x(s)| =1 $ for all $(x,s) \in \Omega$.  Thus, writing $\nabla\Psi=(\nabla_{\bar x}\Psi,\pa_s\Psi )$, we have
   $$\|\mbox{Det } \nabla\Psi\|_{L^\fz(\Omega)} = \|\mbox{Det } (\nabla_{\bar x}\Psi,\pa_s\Psi )\|_{L^\fz(\Omega)}  \lesssim 1.$$
   Combining \eqref{1dim}, we conclude,
  \begin{align*}
     \calL^{n-1}(\{y \in U: & |h(y)| \le 2\delta\})
     = \int_U \chi_{\{|h(y)| \le 2\delta\}}(y) \, dy \\
     & \le  \int_{x \in \pa U^+} \int_{I_x}\chi_{\{  |h(\gz_x(s))| \le 2\dz   \}} (s)\mbox{ $|$Det } \nabla \Psi(\bar x,s)|\, ds \, d\calH^{n-2}(x) \\
      & \lesssim  \int_{x \in \pa U^+} |I_x| \, d\calH^{n-2}(x) \\
      &\lesssim \calH^{n-2}(\pa U^+) \delta/ t \lesssim \delta/ t.
   \end{align*}
  Since
  $  P_{f,g}=  \{ |h| \le 2\delta\}, $
  we arrive at \eqref{Pfg}.

  \medskip

  \emph{Case 2.} Corollary \ref{uniq} (A)(iii) holds, that is,
  \begin{equation}\label{standing2}
    |\nabla_\xi \nabla_\xi h (x)| \ge \frac{1}{3K} \bar t := \frac{1}{3K} \|h\|_{C^2(I_F^{n-1}) } \qquad \mbox{for all } x \in  2U, \xi \in S^{n-2}.
  \end{equation}
  Recall
  $$ \Delta=\Delta(f,g)= \inf_{x \in  U}\{ |f(x)-g(x)| + |\nabla f(x)- \nabla g(x)| \} = \inf_{x \in  U}\{ |h(x)| + |\nabla h(x)| \}.$$

  \medskip
  \emph{Subcase 2-1.} Assume $\Delta(f,g) \ge \frac{\az(K^{-1}/6)^2}{10^5K^4}\bar t= \frac{\az(K^{-1}/6)^2}{10^5K^4} \|f-g\|_{C^2(I^{n-1}_F)}$.

Applying the same argument in \emph{Case 1, Step 1},
instead of \eqref{tlz} and \eqref{lzxtx}, we have
  \begin{equation*}
    \frac{\az(K^{-1}/6)^2}{10^5K^4}\bar t  \le \lambda \le  \lambda_x \le t_x \le \bar t  \quad  x \in \pa U^+.
  \end{equation*}
  In particular,
   \begin{equation*}
     \lambda_x \ge \frac{\az(K^{-1}/6)^2}{10^5K^4} t_x \mbox{ and } t_x \ge \frac{\az(K^{-1}/6)^2}{10^5K^4}\bar t \quad x \in \pa U^+.
   \end{equation*}
  Combining \eqref{ass1}, we have $t_x \ge 4 \frac{10^5K^4}{\az(K^{-1}/6)^2}\delta$ for all $x \in \pa U^+$.
  We apply Lemma \ref{pyz1} (1) to each $h_x$ with $c_1=\frac14 \frac{\az(K^{-1}/6)^2}{10^5K^4} $ and $c_2=\frac{\az(K^{-1}/6)^2}{10^5K^4} \ge 4c_1$. Then use the same argument as in \emph{Case 1, Step 2},
we have $ \calL^{n-1}(P_{f,g})
      \lesssim \delta/ \bar t$ and noting that $t \le \bar t$, we  arrive at \eqref{Pfg}.

\medskip
  \emph{Subcase 2-2.} Assume ${\Delta}(f,g) <  \frac{\az(K^{-1}/6)^2}{10^5K^4}\|h\|_{C^2(I^{n-1}_F)} = \frac{\az(K^{-1}/6)^2}{10^5K^4} \bar t$.

By Corollary \ref{uniq} (B), we may assume $h$ is strictly convex in $2U$ (otherwise we consider $h=g-f$).
  In this subcase, we will deeply use the convexity of $h$ and divide the proof into three steps. In \emph{Step 1}, we first find the unique point $x_M \in 2U$ such that $\nabla(x_M)=0$ by constructing an auxiliary family of functions $\{h_\xi\}$. In \emph{Step 2}, we define another family $\{\bar h_\xi\}$ and show that $\{\bar h_\xi\}$ satisfies Lemma \ref{pyz1} (2). Then, in \emph{Step 3}, we apply Lemma \ref{pyz1} (2) to deduce \eqref{Pfg}.

\medskip
\emph{Step 1}.
  Let $x_{\Delta} \in U$ be the point such that ${\Delta}=|h(x_{\Delta})| + |\nabla h(x_{\Delta})|$.
  Then $|h(x_{\Delta})| \le {\Delta}$ and $|\nabla h(x_{\Delta})| \le {\Delta}$. We will show that there exists a point $x_M$ such that $\nabla h(x_M)=0$ and
  \begin{equation}\label{xm}
    x_M \in \frac65 U \mbox { and } |h(x_M)| \le 10K^2\Delta.
  \end{equation}

  First, if $\nabla h(x_{\Delta})=0$, since Corollary \ref{uniq} (A)(iii) holds for $h$, by Corollary \ref{uniq} (B), we know $x_M=x_{\Delta}$. Thus apparently, \eqref{xm} holds.

  Next, assume $\nabla h(x_{\Delta})  \ne 0$.  For each $\xi \in S^{n-2}$, let $$E_\xi=\{s \in \R : x_{\Delta} + s\xi \in 2U \}$$
  and
  \begin{equation}\label{defh}
    h_\xi(s):= h(x_{\Delta} + s\xi) \quad s \in E_\xi.
  \end{equation}
  Thus by the strict convexity of $h$, we know
  $h_\xi: E_\xi \to \R$ is a $C^2$ strictly convex function for each $\xi \in S^{n-2}$. Note that
  \begin{equation}\label{hxi0}
    h_\xi'(s)= \nabla_\xi h (x_{\Delta} + s\xi) \mbox{  and } h_\xi''(s)= \nabla_\xi \nabla_\xi h (x_{\Delta} + s\xi).
  \end{equation}
In particular, $|h_\xi'(0)| \le |\nabla h (x_{\Delta})| \le \Delta$.

  Since $|\nabla_\xi\nabla_\xi  h(x)| \ge \frac{1}{3K} \bar t$ for all $x \in 2U$ and $\xi \in S^{n-2}$, for each $\xi$, there exists $s_\xi \in E_\xi$ with $|s_\xi| \le 3K{\Delta}/ \bar t < \frac{\az(K^{-1}/6)}{20} < \frac{\diam U}5$ such that
  $ h_\xi'(s_\xi)=0$ where we recall $\diam U > \frac{\az(K^{-1}/6)}{4}$. Moreover, we have
  $$  x_{\Delta}+ s_\xi\xi \in \frac{6}5 U \subset 2U^\circ \quad \xi \in S^{n-2}.$$
  In particular, $s_\xi \notin \pa E_\xi$ and $x_{\Delta}+ s_\xi\xi \notin \pa(2U)$.
  In addition, by the strict convexity of $h_\xi$, we know $s_\xi$ is the unique point in $E_\xi$ which reaches the minimum of $h_\xi$.
  We show that there exists $\xi_0 \in S^{n-2}$ such that
  \begin{equation}\label{mins}
     h(x_{\Delta}+ s_{\xi_0}\xi_0) = \min_{\xi \in S^{n-2}} \{h(x_{\Delta}+ s_{\xi}\xi) \}.
  \end{equation}
  To see \eqref{mins}, it suffices to show that $s_\xi$ continuously depends on $\xi$. Indeed, consider the $C^1$ function, $\phi(\xi,s) =\nabla h(x_{\Delta}+ s\xi) \cdot \xi$.
  Then $\phi(\xi,s_\xi)=0$ for all $\xi \in S^{n-2}$ and $\frac{\pa\phi}{\pa s}(\xi,s)= \nabla_\xi\nabla_\xi(x_{\Delta}+ s_\xi\xi) \ne 0$ at all $(\xi,s_\xi)$. Applying the implicit function theorem for $\phi$, we know $s_\xi$ is a continuous function of $\xi$ and \eqref{mins} holds.

  Define
  $$U_\ez:= \bigcup_{|\xi-\xi_0|\le \ez } \{x_{\Delta}+ s\xi : |s-s_\xi| \le \ez \}.$$
  Choose an $\ez$ small enough such that $U_\ez \subset 2U^\circ$, we know
  for each point $x \in U_\ez$, $x= x_{\Delta}+ s_x\xi_x$ for some $\xi_x$ and $s_x$. Thus
  $$  h(x) = h(x_{\Delta}+ s_x\xi_x) = h_{\xi_x}(s_x) \ge h_{\xi_x}(s_{\xi_x}) =  h(x_{\Delta}+ s_{\xi_x}\xi_x) \ge h(x_{\Delta}+ s_{\xi_0}\xi_0).  $$
  As a consequence, we know $x_{\Delta}+ s_{\xi_0}\xi_0$ reaches the local minimum of $h$ in $U_\ez$. Recalling that $h \in C^2(I_F^{n-1})$, we know $\nabla h(x_{\Delta}+ s_{\xi_0}\xi_0) =0$.


  Let $x_M= x_{\Delta}+ s_{\xi_0}\xi_0$.
  We have
  $\nabla h(x_M)=0$ and $x_M \in \frac65 U$ as desired.
By Corollary \ref{uniq} (B), we know $x_M$ is the unique point such that $\nabla h =0$. 
We also have
$$  |h(x_M)- h(x_{\Delta})| \le \max_{s \in [0,1]}\{|\nabla_e\nabla_e  h((1-s)x_M + s x_{\Delta} )|\}|x_M -x_{\Delta}|^2 \le \bar t(3K{\Delta}/\bar t)^2 \le 9K^2\Delta  $$
where $e = \frac{x_M -x_{\Delta}}{|x_M -x_{\Delta}|}$. Recalling that $|h(x_{\Delta})|\le \Delta$, we have $|h(x_M)| \le 10K^2\Delta$.
Thus \eqref{xm} holds.



\medskip
\emph{Step 2}. We define the family $\{\bar h_\xi\}_{\xi \in S^{n-2}}$ and check $\{\bar h_\xi\}_{\xi \in S^{n-2}}$ satisfies Lemma \ref{pyz1} (2).

Let
$$   I_\xi: = \{s \ge 0 : x_M + s\xi \in 2U\}$$
and
$$  \bar h_\xi (s): = h(x_M +s \xi),  \quad s \in I_\xi, \ \xi \in S^{n-2}.   $$
Since $x_M \in \frac65 U$, we know $\min I_\xi =0$ for all $\xi \in S^{n-2}$. Thus $\bar h_\xi(0)$ is well-defined and $\bar h_\xi(0)=h(x_M)$. Similar to \eqref{hxi0}, we compute
\begin{equation}\label{hxi2}
   \bar h_\xi'(s)= \nabla_\xi h (x_M + s\xi) \mbox{  and } h_\xi''(s)= \nabla_\xi \nabla_\xi h (x_M + s\xi) \quad s \in I_\xi^\circ.
  \end{equation}
Since $\nabla h(x_M)=0$, by \eqref{hxi2} and the standing assumption \eqref{standing2} of   \emph{Case 2}, we have
$$ \lim_{s \to 0+} \bar h_\xi'(s)=\lim_{s \to 0+}\nabla h(x_M+ s \xi) \cdot \xi = \nabla h(x_M) \cdot \xi =0 $$
and
$$\bar h_\xi''(s)= \nabla_\xi\nabla_\xi \bar h(x_M+s\xi) \ge \frac{1}{3K} \bar t \quad s \in I_\xi^\circ. $$
Moreover, let $t_\xi:= \|\bar h_\xi\|_{C^2(I_\xi)}$ for all $\xi \in S^{n-2}$. Then,
it is easy to see $t_\xi=\|\bar h_\xi\|_{C^2(I_\xi)} \le \|h\|_{C^2(I_F^{n-1})} =\bar t$. Thus
\begin{equation}\label{txi2}
  \bar h_\xi''(s) \ge \frac{1}{3K} \|\bar h_\xi\|_{C^2(I_\xi)} =\frac{1}{3K} t_\xi \quad s \in I_\xi^\circ.
\end{equation}
As a consequence, we know $h_\xi$ satisfies the condition of Lemma \ref{pyz1} (2) with constant $c_2=\frac{1}{3K}$ for all $x \in S^{n-2}$. We end this step by letting
$$  \bar \lz:=h(x_M) \equiv \bar h_\xi(0). $$
Recalling \eqref{xm},
we have
\begin{equation}\label{barlz}
   \bar \lz \le 10K^2\Delta.
\end{equation}

\medskip
\emph{Step 3}. Notice that \eqref{txi2} gives
$$ \bar t \ge t_\xi =  \|\bar h_\xi\|_{C^2(I_\xi)} \ge \|\bar h_\xi''\|_{L^\fz(I_\xi)} \ge \frac{1}{3K} \bar t \quad \xi \in S^{n-2}.$$
Thus $t_\xi \sim \bar t$.
 Combining \eqref{ass1}, we have $t_\xi \ge 4\frac{10^9K^7}{\az(K^{-1}/6)^4}\delta$ for all $\xi \in S^{n-2}$.
   Applying Lemma \ref{pyz1} (2) to each $\bar h_\xi$ with $c_1=\frac14\frac{\az(K^{-1}/6)^4}{10^9K^7}$ and $c_2=(3K)^{-1}$, we know
\begin{equation}\label{2dim}
     |\hat I_\xi:=\{s \in I_\xi: |\bar h_\xi(s)| \le 2\delta\}| \lesssim \delta/\sqrt{\bar \lz t_\xi} \sim  \delta/\sqrt{\bar \lz \bar t}
   \end{equation}
   and
   \begin{equation}\label{maxi}
     \max \hat I_\xi \le c_2^{-1}\sqrt{\frac{\bar \lz +\dz}{ \bar t}} = 3K\sqrt{\frac{\bar \lz +\dz}{ \bar t}} \quad \xi \in S^{n-2}.
   \end{equation}

Noting that
$$  2U = \bigcup_{\xi \in S^{n-2}}\{x_M + \xi I_\xi\},$$
we know
$$ \{ y \in 2 U : |h(y)| \le 2\dz \} = \bigcup_{\xi \in S^{n-2}}\{x_M + \xi s  : |\bar h_\xi(s)| \le 2\dz\} = \bigcup_{\xi \in S^{n-2}}\{x_M + \xi \hat I_\xi \}. $$
As a consequence, combining \eqref{maxi}
we have
  \begin{equation}\label{inc1}
    P_{f,g}= \{ y \in  U : |h(y)| \le 2\dz \}  \subset \bigcup_{\xi \in S^{n-2}}\{x_M + \xi \hat I_\xi \} \subset B(x_M, 3K\sqrt{\frac{\bar \lz +\dz}{ \bar t}} ).
  \end{equation}

   Since the parameters $(\xi,s)$ forms a natural polar coordinate centered at $x_M$,
    by considering $\bar\lambda \subset [0,\dz]$ and $\bar\lambda >\dz$ separately,
   in the case $\bar \lz > \dz$,
  noting that
  $$ \max \hat I_\xi \le 3K \sqrt{(\bar\lambda +\dz) / \bar t} \le 6K\sqrt{\bar\lambda / \bar t} \mbox{  and } |\hat I_\xi| \sim \dz/\sqrt{(\bar\lambda +\dz)\bar t} \sim \dz /\sqrt{\bar\lambda \bar t},$$ we deduce
  \begin{align*}
     \calL^{n-1}(P_{f,g})
     &= \int_U \chi_{\{y \in U: |h(y)| \le 2\delta\}}(y) \, dy \\
     & \le  \int_{\xi \in S^{n-2}} \int_{0}^{6K\sqrt{\bar\lambda / \bar t}}\chi_{\{ |\bar h_\xi(s)| \le 2\dz  \}} (s)s^{n-2}\, ds \, d\calH^{n-2}(\xi) \\
     & \lesssim  \int_{\xi \in S^{n-2}} \int_{\hat I_\xi} (\sqrt{\bar\lambda / \bar t})^{n-2}\, ds \, d\calH^{n-2}(\xi) \\
      & \lesssim  \calH^{n-2}(S^{n-2}) |\hat I_\xi|  (\sqrt{\bar\lambda / \bar t})^{n-2}  \\
      &\lesssim \frac{\dz}{\sqrt{\bar\lambda \bar t}} (\sqrt{\bar\lambda / \bar t})^{n-2} =  \frac{\dz\sqrt{\bar\lambda}^{n-3}}{\sqrt{\bar t}^{n-1}} \lesssim \frac{\dz}{\bar t}
   \end{align*}
   where in the last inequality we recall $\bar\lambda = h(x_M) \le \|h\|_{C^2(2U)} \le \bar t$ and $n\ge 3$.

  And in the case $\bar \lz \in [0,\dz]$, noting that
  $$ \max \hat I_\xi \le 3K\sqrt{(\bar\lambda +\dz) / \bar t} \le 6K\sqrt{\dz / \bar t} ,$$
  we deduce that
  \begin{equation}\label{omega}
     P_{f,g} \subset B(x_M, 6K\sqrt{\dz / \bar t}),
   \end{equation}
  which implies
  \begin{align*}
    \calL^{n-1} (P_{f,g}) & \le   \calL^{n-1} [B(x_M, 6K\sqrt{\dz / \bar t} )]
     \lesssim (\delta/ \bar t )^{(n-1)/2} \le \delta/ \bar t .
  \end{align*}

Finally, recalling that
  $\bar t  =\|h\|_{C^2(I_F^{n-1})}  \ge \|h\|_{C^2([0,1]^{n-1})}  \ge t$,
  we arrive at \eqref{Pfg}. The proof is complete.
\end{proof}

The following lemma will be used in Section 4.
\begin{lemma} \label{shape}
  Under the same assumption as in Lemma \ref{dznbhd1}, let $\ez>0$, $0<s \le 1$ and $E \subset [0,1]^{n-1}$ be a $(\dz,n-2+s,\dz^{-\ez})$-set with $|E| \le \dz^{-(n-2+s)}$. Denote by $P_{f,g}$ the orthogonal projection of
  $f^\delta \cap g^\delta$ to $[0,1]^{n-1}$. Then
  \begin{equation}\label{prfg}
   |E \cap P_{f,g}|_\dz \lesssim \delta^{-\epsilon} \frac{\dz^{-(n-2)}}{\|f-g\|_{C^2([0,1]^{n-1})}^s}.
  \end{equation}
\end{lemma}

\begin{proof}
  If $P_{f,g} = \emptyset$ and $\delta \le t:=\|f-g\|_{C^2([0,1]^{n-1})} \le 4[\frac{10^5K^4}{\az(K^{-1}/6)^2}]^2\delta$, it is easy to see \eqref{prfg} holds. In the following we assume $t>4[\frac{10^5K^4}{\az(K^{-1}/6)^2}]^2\delta$.

  Let $\{U\}$ be the dyadic cubes in the above proof. Similar to the proof of the above lemma, to show \eqref{prfg}, it suffices to prove
  \begin{equation}\label{prfg2}
   |E \cap P_{f,g} \cap U|_\dz \lesssim \delta^{-\epsilon} \frac{\dz^{-(n-2)}}{\|f-g\|_{C^2([0,1]^{n-1})}^s}
  \end{equation}
  for each $U$.
  To this end, we fix $U \subset [0,1]^{n-1}$ and write $P_{f,g}=P_{f,g} \cap U$ for simplicity.
  In the following, let all the notations be same as those in the proof of Lemma \ref{dznbhd1}.

If we are in \textit{Case 1} (and also \textit{Subcase 2-1}), then
(by considering $\pa U^+$ in each facet of $\pa U$), we know
$  P_{f,g} = \Psi(\Omega) $  where $\Omega= \cup_{x \in \pa U^+}\{(\bar x, I_x)\} \subset \R^{n-1}$. Recall from \eqref{1dim} that $|I_x| \sim \frac{\dz}{\sqrt{\lz t}} \sim \frac{\dz}{t}$.
Thus $\Omega$ can be covered by $C (\frac{t}{\dz})^{n-2}$ many balls $B^{n-1}(x_i,\frac{\dz}{t})$ of radius $\frac{\dz}{t}$.
Since $\|\mbox{Lip }\Psi\|_{L^\fz(\Omega)} \lesssim 1$, we know
$P_{f,g}$ can be covered by $\bar C (\frac{t}{\dz})^{n-2}$ many balls $B^{n-1}(y_i,\frac{\dz}{t})$ of radius $\frac{\dz}{t}$.

Since $\frac{\dz}{t} \ge \dz$ and $E$ is a $(\dz,n-2+s,\dz^{-\ez})$-set, we have
\begin{align*}
  |E \cap P_{f,g}|_\dz & \le \sum_{i=1}^{\bar C (t\dz^{-1})^{n-2}} |E \cap B^{n-1}(y_i, \frac{\dz}{t})|_\dz \\
   & \le \bar C(t\dz^{-1})^{n-2}\delta^{-\epsilon}(\frac{\dz}{t})^{n-2+s} |E|  \\
   & \lesssim \delta^{-\epsilon} \frac{\dz^{-(n-2)}}{t^s}=\delta^{-\epsilon} \frac{\dz^{-(n-2)}}{\|f-g\|_{C^2([0,1]^{n-1})}^s}
\end{align*}
which gives \eqref{prfg}.

\medskip
  Assume we are in \textit{Subcase 2-2}. Recall $\bar t:= \|f-g\|_{C^2(I_F^{n-1})} \ge t=\|f-g\|_{C^2([0,1]^{n-1})}$.

  We first consider the case $\bar \lz \in [0,\dz] $. In this case, noting that $\sqrt{(\bar\lambda +\dz) / \bar t} \le 2\sqrt{\dz / \bar t}$ and recalling \eqref{omega}, we know
  we know $  P_{f,g} \subset B(x_M, 6K\sqrt{\dz / \bar t})$.
  As a result, noting that $n-2+s \ge 1+s \ge 2s$,  we have
\begin{align*}
  |E \cap P_{f,g}|_\dz & \le  |E \cap B(x_M, 6K\sqrt{\dz / \bar t})|_\dz \\
   & \le \delta^{-\epsilon}(6K\frac{\dz}{\bar t })^{n-2+s} |E|  \\
   & \lesssim \delta^{-\epsilon} \frac{\dz^s}{\bar t^s}\dz^{-(n-2+s)} \le \delta^{-\epsilon} \frac{\dz^{-(n-2)}}{\|f-g\|_{C^2([0,1]^{n-1})}^s}
\end{align*}
which gives \eqref{prfg}.

   For the case $\bar \lz > \dz $, noting that $\sqrt{(\bar\lambda +\dz) / \bar t} \le 2\sqrt{\bar\lambda  / \bar t}$ and recalling \eqref{inc1},
  we know $  P_{f,g} \subset B(x_M, 6K\sqrt{\bar\lambda  / \bar t})$. Moreover, recalling the standing assumption in \textit{Subcase 2-2} that $\Delta ={\Delta}(f,g) <  \frac{\az(K^{-1}/6)^2}{10^5K^4}\|h\|_{C^2(I^{n-1}_F)} = \frac{\az(K^{-1}/6)^2}{10^5K^4} \bar t$ and inequality \eqref{barlz} that $\bar \lz \le 10K^2 \Delta$, we know
  $$ 6K\sqrt{\bar\lambda  / \bar t} \le 6K \frac{\az(K^{-1}/6)}{100K} \le \frac{\az(K^{-1}/6)}{16} \le \frac1{4}\diam U $$
  where we recall $\frac{\az(K^{-1}/6)}4 < \diam U < \frac{\az(K^{-1}/6)}2$.
  By further recalling $x_M \in \frac65 U$, we deduce that
  $$ P_{f,g} \subset B(x_M, 6K\sqrt{\bar\lambda  / \bar t}) \subset \frac32 U \subset 2U. $$

   On the other hand,
   by Corollary \ref{uniq} (B), we
   can assume $h$ is strictly convex in $2U$. Thus
   $$  \Lambda_{2\dz}(h):=\{x \in 2U : h(x) \le 2\dz\}$$
   is a convex set. Since $x_M$ is the unique point in $2U$ such that $\nabla h(x) =0$, we know $x_M$ is the unique point reaches the minimum of $h$.
   And we may assume $h(x_M) < 2\dz$ (otherwise $h(x) > h(x_M) \ge 2\dz$ for all $x \in 2U$ and $P_{f,g}=\{x_M\}$). In particular,
   $x_M \notin \pa \Lambda_{2\dz}(h)= \{x \in 2U : h(x) =2\dz\}$. Hence for all $x \in \pa \Lambda_{2\dz}(h)$, we have $\nabla h(x) \ne 0$. Applying the preimage theorem for $h$, we know $\pa \Lambda_{2\dz}(h)$ is an $(n-2)$-dimensional manifold. Moreover, $\{x \in 2U : h(x) <2\dz\}$ is a nonempty open set and $\calL^{n-1}(\Lambda_{2\dz}(h))>0$.

   Noting that $\Lambda_{2\dz}(h) \subset B(x_M, 6K\sqrt{\bar\lambda  / \bar t})$ (since $\pa \Lambda_{2\dz}(h) \subset P_{f,g} \subset B(x_M, 6K\sqrt{\bar\lambda  / \bar t})$ and $\Lambda_{2\dz}(h)$ is convex),
   applying Lemma \ref{convset}, we have
   \begin{equation}\label{conv2}
      \calH^{n-2}(\pa \Lambda_{2\dz}(h)) \le \calH^{n-2}(\pa B(x_M, 6K\sqrt{\bar\lambda  / \bar t})) \lesssim \sqrt{\bar\lambda  / \bar t}^{n-2}.
   \end{equation}
   For any $x \in P_{f,g}= \{x \in 2U : |h(x)| \le 2\dz\} \subset \Lambda_{2\dz}(h)$, we know $x=x_M+\xi_x \tau$ for some $\xi_x \in S^{n-2}$ and $\tau \in \hat I_{\xi_x}$. We also know that there exists $\tau' \in \hat I_{\xi_x}$ such that $x_M+\xi_x \tau' \in \pa \Lambda_{2\dz}(h)= \{x \in 2U : h(x) = 2\dz\}$. Using $|\tau -\tau'| \le |\hat I_{\xi_x}| \lesssim\dz /\sqrt{ \bar\lambda\bar t}$, we know
   \begin{equation*}
     \dist(x, \pa \Lambda_{2\dz}(h) )  \le |\tau -\tau'| \lesssim\dz /\sqrt{ \bar\lambda\bar t}, \quad x \in P_{f,g}.
   \end{equation*}
   Consequently, we know $P_{f,g} \subset [\pa \Lambda_{2\dz}(h)]^{C_1\dz /\sqrt{ \bar\lambda\bar t}}$ for some constant $C_1>0$ where $[\pa \Lambda_{2\dz}(h)]^a$ is the $a$-neighbourhood of $\pa \Lambda_{2\dz}(h)$.
   Combining \eqref{conv2} and by an elementary computation,
   we know
 $P_{f,g}$ can be covered by $C_2 (\frac{\bar \lz}{\dz})^{n-2}$ many balls $B^{n-1}(x_i,\dz/\sqrt{\bar \lz \bar t})$ of radius $\dz/\sqrt{\bar \lz \bar t}$ for some constant $C_2>0$.

Since $\dz/\sqrt{\bar \lz \bar t} \ge \dz$ and $E$ is a $(\dz,n-2+s,\dz^{-\ez})$-set, recalling that $\bar\lz \le \bar t$ and $t = \|f-g\|_{C^2([0,1]^{n-1})}\le \bar t$, we have
\begin{align*}
  |E \cap P_{f,g}|_\dz & \le \sum_{i=1}^{C_2(\bar \lz\dz^{-1})^{n-2}} |E \cap B^{n-1}(y_i, \dz/\sqrt{\bar \lz \bar t})|_\dz \\
   & \le C_2(\bar \lz \dz^{-1})^{n-2}\delta^{-\epsilon}(\dz/\sqrt{\bar \lz \bar t})^{n-2+s} |E|  \\
   & = C_2\delta^{-\epsilon} \frac{\bar \lz^{(n-2-s)/2}}{\bar t^{(n-2+s)/2}} \dz^{s}|E| \\
   & \lesssim \delta^{-\epsilon} \frac{\dz^{-(n-2)}}{\bar t^s} \le \delta^{-\epsilon} \frac{\dz^{-(n-2)}}{\|f-g\|_{C^2([0,1]^{n-1})}^s}
\end{align*}
which gives \eqref{prfg}.

The proof is complete.
\end{proof}

We are able to show Theorem \ref{maxfcn}.

\begin{proof}[Proof of Theorem \ref{maxfcn}]
Let $2\log (1/\dz_0)< \dz_0^{-\ez}$.

   For each $f \in F$ and $i \in  \mathbb{N}$, define
   $$  F_i(f):= \{ g \in F : \|f-g\|_{C^2([0,1]^{n-1})}  \in (2^{-(i+1)},2^{-i}]  \} . $$
   Note that if $i \ge \log (1/\dz)$, then $F_i(f) =\{f\}$ due to $F$ being $\delta$-separated.
   Since $F$ is $(\delta,t,\dz^{-\ez})$-set, we have $|F_i(f)| \le \dz^{-\ez}2^{-it}|F|$.
   Using Lemma \ref{dznbhd1}, we have
   \begin{align}\label{l22}
    \sum_{f \in F}\sum_{g \in F, g \ne f} \calL^n(f^\dz \cap g^\dz) &= \sum_{f \in F}\sum_{i =1}^{\log (1/\dz)}\sum_{g \in F_i(f)}  \calL^n(f^\dz \cap g^\dz) \notag\\
    &\lesssim \dz^{-\ez} |F| \sum_{f \in F}\sum_{i=1}^{\log (1/\dz)}2^{-it} \frac{\dz^2}{2^{-(i+1)}} \notag \\
    & \le  2\dz^{-\ez}|F|\dz^2\sum_{f \in F} \sum_{i=1}^{\log (1/\dz)} 2^{(1-t)i} \notag\\
    & \le 2\dz^{-\ez}|F|^2 \dz^{1+t} \log (1/\dz)
     \le \dz^{-2\ez}|F|^2 \dz^{1+t}
   \end{align}
   where in the second last inequality we note that $\sum_{i=1}^{\log (1/\dz)} 2^{(1-t)i} \le \sum_{i=1}^{\log (1/\dz)} 2^{(1-t)\log (1/\dz) } = \log (1/\dz) \dz^{t-1} $.

   Notice that
   \begin{equation}\label{l21}
     \int_{[0,1]^n}  (\sum_{f \in F} \chi_{f^\dz})^2 \, dx = \sum_{f \in F}\sum_{g \in F}\int_{[0,1]^n}  \chi_{f^\dz}\chi_{g^\dz} \, dx = \sum_{f \in F}\sum_{g \in F}  \calL^n(f^\dz \cap g^\dz).
   \end{equation}
   Combining \eqref{l22} and \eqref{l21}, we have
   \begin{align*}
     \int_{ [0,1]^n}  (\sum_{f \in F} \chi_{f^\dz})^2 \, dx & = \sum_{f \in F} \calL^n(f^\dz )  + \sum_{f \in F}\sum_{g \in F, g \ne f}  \calL^n(f^\dz \cap g^\dz) \\
      & \lesssim |F| \dz +\dz^{-2\ez}|F|^2 \dz^{1+t} \lesssim \dz^{-2\ez}|F|^2 \dz^{1+t} ,
   \end{align*}
   which gives \eqref{l2} as desired.
\end{proof}

\section{Furstenberg sets of hypersurfaces}

In this section, we focus on the dimension lower bound of Furstenberg sets of hypersurfaces. Precisely, we prove Theorem \ref{thmconfig} and use Theorem \ref{thmconfig} to show Theorem \ref{dimf}.

First, we make some reductions before showing Theorem \ref{dimf}.
Let $E \subset [0,1]^{n}$ be a $(s, t)_{n-1}$-Furstenberg set
 with parameter set $F \subset C^2([0,1]^{n-1}, [0,1])$ with $\Hd F \geq t$.
Note that if $0 < t \le s \le 1$, then the desired bound is $n-2+s+t$. Otherwise, the desired bound is $n-2+2s$. Noticing that once the bound $n-2+s+t$ is proved, the bound $n-2+2s$ immediately follows, since if $t \ge s$, any $(n-2+s,t)_{n-1}$-Furstenberg set contains an $(n-2+s,s)_{n-1}$-Furstenberg set. It suffices to show the bound is $n-2+s+t$ when $0 < t \le s \le 1$.

In the following of this section, we assume $0 < t \le s \le 1$.
To show Theorem \ref{thmconfig}, we also make some reductions. Let $\Omega$ be the $(\delta,s,t,\delta^{-\epsilon})_{n-1}$-configuration in Theorem \ref{thmconfig}. Write $F := \pi_{1}(\Omega) \subset C^2([0,1]^{n-1})$, and $E(f) = \{x \in \R^{n} : (f,x) \in \Omega\} \subset \mbox{graph}(f)$.
By replacing $F$ and $E(f)$ by maximal $\delta$-separated subsets, we may assume that $F$, $E(f)$, and $\Omega$ are finite and $\delta$-separated to begin with. Furthermore, $F$ contains a $(\delta,t,\delta^{-2\epsilon})$-subset $\bar{F} \subset F$ of cardinality $\frac12 \delta^{2\epsilon}\delta^{-t} \le |\bar{F}| \leq \delta^{-t}$ by \cite[Lemma 2.7]{2021arXiv210603338O}.
 Recalling $E(f)$ is a $(\delta,n-2+s,\delta^{-\epsilon})$-set implies that $M = |E(f)| \ge \delta^{\epsilon-(n-2+s)}$. If $M> \delta^{-(n-2+s)}$, we can always find a $(\delta,n-2+s,\delta^{-2\epsilon})$-set $\bar E(f) \subset E$ such that $\bar M \equiv|\bar E(f)| \ge \delta^{2\epsilon-(n-2+s)}$ for all $f\in \bar{F}$.
 Then $\bar{\Omega} := \{(f,x) : f\in \bar{F} \text{ and } x \in \bar E(f)\}$ remains a $(\delta,s,t,\delta^{-2\epsilon})_{n-1}$-configuration with $|\bar E(f)| \equiv \bar M$.
Apparently,  it suffices to prove Theorem \ref{thmconfig} for this sub-configuration, so we may assume that
\begin{equation}\label{bdM}
 |F| \in [ \frac12\delta^{2\epsilon-t},\delta^{-t}] \text{ and } M \in [ \delta^{2\epsilon-(n-2+s)},\delta^{-(n-2+s)}]
\end{equation}
to begin with. Finally,  $F$ is a cinematic family implies that $\|f\|_{C^2([0,1]^{n-1})} \le K$ for all $f \in F$.

The above notions and assumptions will be fixed throughout this section. We need several auxiliary lemmas.

\begin{lemma} \label{csh}
Let $A_1, \cdots, A_m \subset \R^n$ be $m$ Lebesgue measurable subsets of finite Lebesgue measure. Then
$$  \calL^n(\bigcup_{i=1}^m A_i)  \ge \frac{[\sum_{i=1}^{m} \calL^n(A_i) ]^2}{\sum_{i,j=1}^m\calL^n( A_i\cap A_j)} . $$
\end{lemma}

\begin{proof}
  Applying Cauchy-Schwarz inequality to $\sum_{i=1}^{m}\chi_{ A_i}$ and noting that $\|\sum_{i=1}^{m}\chi_{ A_i}\|_{L^2(\R^n)} =\sum_{i,j=1}^m\calL^n( A_i\cap A_j)$, we have
  $$ \|\sum_{i=1}^{m}\chi_{ A_i}\|_{L^1(\R^n)} = \int_{\spt (\sum_{i=1}^{m}\chi_{ A_i})} |\sum_{i=1}^{m}\chi_{ A_i}(x)| \, d \calL^n  (x) \le  \calL^n(\bigcup_{i=1}^m A_i)^{1/2} \|\sum_{i=1}^{m}\chi_{ A_i}\|_{L^2(\R^n)}    $$
  as desired.
\end{proof}

In the remaining part of this section, for any set $A \subset [0,1]^n$, let $P_{A}$ be the projection of $A$ to the first $n-1$ coordinates.

\begin{lemma} \label{pe}
There exist $\epsilon = \ez(s,n,K), \bar\delta_{0}=\bar\delta_{0}(\ez) \in (0,\tfrac{1}{2}]$ such that the following holds for all $\delta \in (0,\bar\delta_{0}]$. Let $F$ and $E(f)$ be as above.
  For any $f \in F$, $P_{E(f)}$ is a $(\delta,n-2+s,\delta^{-3\epsilon})$-set.
\end{lemma}

\begin{proof}
Let $\bar\dz_0^{-\ez} > 4^{n-2+s}(2K)^{n-1} $.
Note that $P_{E(f)} \subset [0,1]^{n-1} \subset \R^{n-1}$. We need to check for any ball $B^{n-1}(x,r) \subset \R^{n-1}$ with $r\ge \dz$,
$$   |P_{E(f)} \cap B^{n-1}(x,r)|_{\dz} \le \delta^{-3\epsilon} r^{n-2+s}|P_{E(f)}|_\dz.  $$
Since $\|f\|_{C^2([0,1]^{n-1})} \le K$, graph$(f)$ is a $K$-Lipschitz graph. Thus by $E(f)$ being $\dz$-separated, we know $P_{E(f)}$ is at least $(2K)^{-1}\delta$-separated and $|P_{E(f)}|_\dz \ge (2K)^{-(n-1)}|P_{E(f)}|= (2K)^{-(n-1)}|E(f)|=(2K)^{-(n-1)}M$. Thus it suffices to show
\begin{equation}\label{p2}
   |P_{E(f)} \cap B^{n-1}(x,r)|_{\dz} \le \delta^{-3\epsilon}r^{n-2+s}(2K)^{-(n-1)}M \quad x \in \R^{n-1}, \ r \ge \dz.
\end{equation}
In the following, for any $y \in \R^n$, we write $y=(\bar y, y_n)$ where $\bar y = (y_1, \cdots, y_{n-1})$.
Notice that
\begin{equation}\label{p1}
  P_{E(f)} \cap B^{n-1}(x,r) = \{ \bar y : y \in E(f), \ \bar y \in B^{n-1}(x,r) \}.
\end{equation}
Since graph$(f)$ is a $K$-Lipschitz graph, we know that the set
$\{y \in E(f): \bar y \in B^{n-1}(x,r)\}$ can be covered by $K$ balls $B^n(z_i,4r) \subset \R^n$ with $\bar z_i =x$ for all $i=1,\cdots ,K$. Using that $E(f)$ is a $(\delta,n-2+s,\delta^{-2\epsilon})$-set and $|E(f)|=M$, we deduce that
$$ |\{y \in E(f): \bar y \in B^{n-1}(x,4r)\}|_{\dz} \le \sum_{i=1}^K |\{y \in E(f): y \in B^{n}(z_i,4r)\}|_{\dz} \le  \delta^{-2\epsilon} (4r)^{n-2+s}M. $$
Recalling \eqref{p1}, we know
$$  |P_{E(f)} \cap B^{n-1}(x,4r)|_\dz \le |\{y \in E(f): \bar y \in B^{n-1}(x,4r)\}|_{\dz} \le  \delta^{-2\epsilon} (4r)^{n-2+s}M $$
and by $4^{n-2+s}(2K)^{n-1} < \delta^{-\epsilon}$, we conclude \eqref{p2}.
\end{proof}

\begin{lemma} \label{enbhd} For any $\ez>0$, let $\bar \dz_0$ be as in Lemma \ref{pe}. For all $\delta \in (0, \bar \delta_{0}]$,
  let $E^\dz(f)$ be the $\dz$-neighbourhood of $E(f)$ for all $f \in F$. Then
  \begin{equation}\label{3ez}
    \calL^n(E^\dz(f) \cap E^\dz(g)) \lesssim \delta^{-3\epsilon} \frac{\dz^2}{\|f-g\|_{C^2([0,1]^{n-1})}^s}.
  \end{equation}
\end{lemma}

\begin{proof}
Write $\|f-g\|_{C^2([0,1]^{n-1})} = d$.
Noting that $\dz$-balls in $\R^n$ has measure $\sim \dz^n$, it suffices to show,
   $$ |E^\dz(f) \cap E^\dz(g)|_{\dz} \lesssim \delta^{-3\epsilon} \frac{\dz^{-(n-2)}}{d^s}. $$
   Also, noticing $E^\dz(f) \cap E^\dz(g) \subset f^\dz \cap g^\dz$ has vertical thickness $\lesssim\dz$, we know
   $$  |P_{E^\dz(f) \cap E^\dz(g)}|_{\dz} \sim |E^\dz(f) \cap E^\dz(g)|_{\dz}. $$
   Observe that
   $$ P_{E^\dz(f) \cap E^\dz(g)} \subset P_{E^\dz(f)} \cap P_{f^\dz \cap g^\dz}, $$
   it further reduces to showing
   \begin{equation}\label{pef}
     |P_{E(f)} \cap P_{f^\dz \cap g^\dz}|_\dz \lesssim \delta^{-3\epsilon} \frac{\dz^{-(n-2)}}{d^s}.
   \end{equation}

   Applying Lemma \ref{pe} that $P_{E(f)}$ is a $(\delta,n-2+s,\delta^{-3\epsilon})$-set, recalling that $M \le \delta^{-(n-2+s)}$ and using Lemma \ref{shape} with $E=P_{E(f)}$ (where $P_{f^\dz \cap g^\dz}$ is denoted by $P_{f,g}$ in Lemma \ref{shape}), we deduce \eqref{pef}
   as desired.
\end{proof}

We are able  to prove Theorem \ref{thmconfig}. The proof is partially motivated by \cite[Appendix A]{MR4382473}.

\begin{proof}[Proof of Theorem \ref{thmconfig}]
  Let $C$ be the implicit constant in \eqref{3ez} and
  $ \bar \dz_0(\ez)$ be as in Lemma \ref{enbhd}. The relationship between $\dz_0=\dz_0(\ez)$ and $\ez$ in this theorem is
  \begin{equation}\label{4ez}
    \dz_0(\ez) < \bar\dz_0 \mbox{ and } 2C\log(1/\dz_0) < \frac12\dz_0^{-\ez}.
  \end{equation}

  Let $\mathcal{E}^\dz$ be the $\dz$-neighbourhood of $\mathcal{E}$. Then it suffices to show $\calL^n(\mathcal{E}^\dz) \geq \delta^{12\ez +2- t-s}$.
  Observe that
  $$  \mathcal{E}^\dz =   \bigcup_{f \in F} E^\dz(f).    $$
  Applying Lemma \ref{csh}, we have
  $$ \calL^n(\mathcal{E}^\dz) =  \calL^n(\bigcup_{f \in F} E^\dz(f)) \ge \frac{[\sum_{f \in F}\calL^n( E^\dz(f))]^2}{\sum_{f ,g\in F}\calL^n( E^\dz(f)\cap  E^\dz(g))}.  $$
  Recalling \eqref{bdM} and noting that
  $ \calL^n(E^\dz(f)) = M \dz^n \ge \dz^{2\ez} \dz^{2-s}$ for all $f \in F$ and $|F|\ge \frac12 \dz^{2\ez} \dz^{-t} \ge \dz^{3\ez} \dz^{-t}$, we deduce that
  $$ [\sum_{f \in F}\calL^n( E^\dz(f))]^2 \ge [ \dz^{5\ez +2 -t-s} ]^2.  $$
  It suffices to show
  \begin{equation}\label{e2}
    \sum_{f ,g\in F}\calL^n( E^\dz(f)\cap  E^\dz(g)) \le \dz^{-6\ez}\dz^{2-t-s}
  \end{equation}

  Similar to the proof of Theorem \ref{maxfcn}, for each $f \in F$ and $i \in  \mathbb{N}$, define
   $$  F_i(f):= \{ g \in F : \|f-g\|_{C^2([0,1]^{n-1})}  \in (2^{-(i+1)},2^{-i}]  \} . $$
   Note that if $i \ge \log (1/\dz)$, then $F_i(f) =\{f\}$ due to $F$ being $\delta$-separated.
   Since $F$ is $(\delta,t,\dz^{-2\ez})$-set, we have $|F_i(f)| \le \dz^{-2\ez}2^{-it}|F|$.
   Using Lemma \ref{enbhd}, we have
   \begin{align}\label{l23}
    \sum_{f \in F}\sum_{g \in F, f\ne g}  \calL^n( E^\dz(f)\cap  E^\dz(g)) &= \sum_{f \in F}\sum_{i =1}^{\log (1/\dz)}\sum_{g \in F_i(f)}  \calL^n( E^\dz(f)\cap  E^\dz(g)) \notag\\
    &\le C\dz^{-5\ez} |F| \sum_{f \in F}\sum_{i=1}^{\log (1/\dz)}2^{-it} \frac{\dz^2}{2^{-(i+1)s}} \notag \\
    & \le  2C\dz^{-5\ez}|F|\dz^2\sum_{f \in F} \sum_{i=1}^{\log (1/\dz)} 2^{(s-t)i} \notag\\
    & \le 2C\dz^{-5\ez}|F|^2 \dz^2 \log (1/\dz)\dz^{t-s}
     \le \frac12 \dz^{-6\ez}\dz^{2-t-s}.
   \end{align}
   where in the second last inequality we note that $s\ge t$ and $\sum_{i=1}^{\log (1/\dz)} 2^{(s-t)i} \le \log (1/\dz) 2^{(s-t)\log (1/\dz) } = \log (1/\dz) \dz^{t-s} $, and in the last inequality we recall $|F| \le \dz^{-t}$ and note $\dz^{-5\ez}|F|^2 \le \dz^{-5\ez-2t}$ and \eqref{4ez}.
   Finally, noticing that
   $$  \sum_{f ,g\in F}\calL^n( E^\dz(f)\cap  E^\dz(g)) = \sum_{f \in F} \calL^n( E^\dz(f))  + \sum_{f \in F}\sum_{g \in F, g \ne f}  \calL^n(E^\dz(f)\cap  E^\dz(g))  $$
   and
   $$  \sum_{f \in F} \calL^n( E^\dz(f)) \le|F|  \dz^n M \le \dz^{2-s-t} \le \frac12 \dz^{-6\ez}\dz^{2-s-t},  $$
   we conclude \eqref{e2} as desired.
\end{proof}


We prove Theorem \ref{dimf} using Theorem \ref{thmconfig}.

\begin{proof}[Proof of Theorem \ref{dimf}]

  Fix $0 < t \leq s \leq 1$
and moreover, fix $t' \in [t/2,t]$ and $t' \leq s' < s$. Since $\mathcal{H}^{t'}_{\infty}(F) > 0$, there exists $\alpha = \alpha(F,t') >0$ and $F_1 \subset F$ such that $\calH^{t'}_\infty(F_1)> \alpha$,
where
\begin{equation}\label{lbdd11}
  F_1 := \{ f\in F : \calH^{n-2+s'}_\infty(E \cap \mbox{graph}(f))> \alpha\}.
\end{equation}
This follows from the sub-additivity of Hausdorff content.

Next, let $\ez(\bar s,\bar t)>0$ be the constant such that Theorem \ref{thmconfig} holds for $\bar s,\bar t$. Define $\ez:= \min_{\bar t \in [t',t], \bar s \in [s',s]}\{\ez(\bar s,\bar t)\}$ (where we note that $\ez(\bar s,\bar t)>0$ whenever $\bar s,\bar t>0$).
We further choose $k_0 = k_{0}(\alpha,\epsilon) = k_0(F,t',\epsilon) \in \N$ satisfying
\begin{equation}\label{para31}
  \alpha>\sum_{k=k_0}^{\infty} \frac{1}{k^2} \quad \text{and} \quad k_{0}^{2} \leq \frac{2^{\epsilon k_{0}}}{C},
\end{equation}
where $C \geq 1$ is a constant to be determined later. Let $\calU = \{ D(x_i,r_i)\}_{i \in\mathcal{I}}$ be an arbitrary cover of $E$
by dyadic $r_i$-cubes with $r_i \le 2^{-k_0}$ and
$E \cap D(x_i,r_i) \ne \emptyset$ for all $i \in \mathcal{I}$.
For $k \ge k_0$, write
$$ \mathcal{I}_k:= \{i \in \mathcal{I} :  r_i = 2^{-k} \} \quad \mbox{and} \quad E_k:= \{\cup D(x_i,r_i) : i \in \mathcal{I}_k \}.$$
By the pigeonhole principle and \eqref{para31} we deduce that for each $f \in F_1$, there exists $k(f) \ge k_0$ such that
$$\calH^{n-2+s'}_\infty( E \cap \mbox{graph}(f) \cap E_{k(f)})> k(f)^{-2} .$$
Using pigeonhole principle again we obtain that there exists $k_1 \ge k_0$ such that
\begin{equation}\label{lbdd21}
  \calH^{t'}_\infty(F_2)> k_1^{-2}
\end{equation}
 where $ F_2 := \{f \in F_1 : k(f) = k_1\}.$ By the construction of $F_2$, we have
\begin{displaymath}
\calH^{n-2+s'}_\infty(\mbox{graph}(f)\cap F_{k_1})\ge \calH^{n-2+s'}_\infty(E \cap \mbox{graph}(f) \cap E_{k_1})> k_1^{-2}, \qquad f \in F_{2}.
\end{displaymath}
Write $\delta =2^{-k_1}$.
By \eqref{lbdd21} and \cite[Lemma 2.7]{2021arXiv210603338O}, we know that there exists a $\delta$-separated $(\delta,t',Ck_{1}^{2})$-set $P \subset F_2$ satisfying $(k_{1}^{-2}/C)\delta^{-t'} \leq |P| \leq \delta^{-t'}$. Since $P \subset F_{2}$, we have
\begin{equation}\label{lbdd31}
  \calH^{n-2+s'}_\infty(\mbox{graph}(f)\cap E_{k_1})> k_1^{-2}, \qquad f \in P.
\end{equation}
 Applying \cite[Lemma 2.7]{2021arXiv210603338O} again to $\mbox{graph}(f)\cap E_{k_1}$, $f \in P$, we obtain $\delta$-separated $(\delta,n-2+s',Ck_{1}^{2})$-sets $E(f) \subset \mbox{graph}(f)\cap E_{k_1}$ such that
\begin{equation*}
  |E(f)| \equiv M \geq (k_{1}^{-2}/C)\delta^{-(n-2+s')} \stackrel{\eqref{para31}}{\geq} \delta^{\ez -(n-2+s')}, \qquad f \in P.
\end{equation*}
By \eqref{para31}, $P$ is a $(\delta,t',\delta^{-\epsilon})$-set, and each $E(f)$ is a $(\delta,n-2+s',\delta^{-\epsilon})$-set.
Therefore,
$$ \Omega:= \{(f,x) : f \in P \text{ and } x \in E(f)\} $$
is a $(\delta,s',t',\delta^{-\epsilon},M)_{n-1}$-configuration. Recall that $\epsilon \leq \epsilon(t')$ by the definition of $\epsilon$. Letting
$$ \calE:= \bigcup_{f \in P} E(f)  $$
and applying Theorem \ref{thmconfig}, we deduce that $|\mathcal{E}|_{\delta} \geq  \delta^{16\ez - (n-2)-s'-t'}$.

Since $E(f) \subset E_{k_1}$ for each $f \in P$, we have $\calE \subset E_{k_1}$, which implies
 $$|\mathcal{I}_{k_1}| = |E_{k_1}|_\delta \geq  |\mathcal{E}|_{\delta} \geq \delta^{16\ez - (n-2)-s'-t'}.$$
 Then
$$ \sum_{i \in \mathcal{I}} r_i^{  (n-2)+s'+t'-16\ez} \ge \sum_{i \in \mathcal{I}_{k_1}} r_i^{  (n-2)+s'+t'-16\ez} = \delta^{  (n-2)+s'+t'-16\ez}|\mathcal{I}_{k_1}| \geq 1. $$
As the covering was arbitrary, we infer that $\Hd E \ge   (n-2)+s'+t'-16\ez$.
Sending $s' \nearrow s$, $t' \nearrow t$, and $\ez \searrow 0$, we arrive at the desired result.
\end{proof}

\section{Proof of Theorem \ref{mainproj}}

We prove Theorem \ref{mainproj2} and Theorem \ref{mainproj} in this section. We first use Theorem \ref{mainproj2} to show Theorem \ref{mainproj} by verifying Theorem \ref{mainproj} is a special case of Theorem \ref{mainproj2}. To this end, we recall the definition of principle curvature.

\begin{definition}[Principal Curvature]
  Let $M$ be an $n$-dimensional $C^2$-manifold and $\Sigma \subset M$ be an $m$-dimensional $C^2$-manifold ($m < n$). Let $\nu \in C^2(TM)$ be a normal vector field, i.e. for each $x \in \Sigma$, $\nu(x) \perp T_x\Sigma$.
  Then the $m$ real eigenvalues $\kappa_1 , \cdots \kappa_m$ of $II_x(\nu)$
  are called {\em principal curvatures} of $\Sigma$ at $x$ (with respect to $\nu$) and the corresponding eigenvectors are called {\em principal directions} where $II_x(\nu)$ is the second fundamental form (of $\nu$), i.e.
  $$ II_x(\xi,\eta):= -\langle  \nabla_\xi \nu, \eta\rangle \quad \xi,\eta \in T_x\Sigma  $$
\end{definition}

We relate the ``non-degenerate'' condition for surfaces to principal curvatures.

\begin{lemma} \label{equiv1}
   Assume $\Sigma \subset \R^{n+1}$ satisfies Definition \ref{nond1}(i). Then $\Sigma$ satisfies Definition \ref{nond1}(ii) if and only if all principal curvatures $\kappa_1 , \cdots \kappa_{n-1}$ are not vanishing and have same sign for all $x \in \Sigma$.
\end{lemma}
\begin{proof}
  Fix any $x \in \Sigma$. Let $\{e_1,\cdots,e_{n-1}\} \subset T_x\Sigma$ be the principal directions and $\gz_i : [-1,1] \to \Sigma$ such that $\gz_i(0)=x$ and $\dot \gz_i(0) = e_i$, $i=1,\cdots n-1$. Thus by definition, the $n-1$ principal curvatures are $\kz_i = -\langle  \nabla_{e_i} \nu, e_i\rangle $,
  $i = 1, \cdots, n-1$.

   If $\Sigma$ satisfies Definition \ref{nond1}(ii), then we know
   $$\kz_i\kz_j=\langle  \nabla_{e_i} \nu, e_i\rangle \langle  \nabla_{e_i} \nu, e_i\rangle= \langle  \nabla_{\dot \gz_i(0)} \nu, \dot \gz_i(0)\rangle \langle  \nabla_{\dot \gz_j(0)} \nu, \dot \gz_j(0)\rangle >0 \quad 1 \le i ,j \le n-1.   $$
   Hence all $\kz_i$ are not vanishing and have same sign.

   Conversely, let $\gz_k : [-1,1] \to \Sigma$ such that $\gz_k(0)=x$ and $\dot \gz_k(t) \ne 0$ for all $t \in [-1,1]$, $k=1,2$.
   Write $\dot \gz_k(t)= \sum_{i=1}^{n-1} a_k^i(t) e_i $, $k=1,2$.
   We have
   \begin{align*}
     \langle  \nabla_{\dot \gz_1(0)} \nu, \dot \gz_1(0)\rangle \langle  \nabla_{\dot \gz_2(0)} \nu, \dot \gz_2(0)\rangle &  = [\sum_{i=1}^{n-1}(a_1^i)^2\langle  \nabla_{e_i} \nu, e_i\rangle] [\sum_{i=1}^{n-1}(a_2^i)^2\langle  \nabla_{e_i} \nu, e_j\rangle]\\
      & =\sum_{i,j=1}^{n-1}(a_1^i)^2(a_2^j)^2\kz_i\kz_j >0,
   \end{align*}
  That is, $\Sigma$ satisfies Definition \ref{nond1}(ii). The proof is complete.
\end{proof}

We are able to show Theorem \ref{mainproj}.

\begin{proof}[Proof of Theorem \ref{mainproj}]
   Since $\Sigma \subset S^n$, for all $x \in \Sigma$, we know there exists a unique $\nu(x)\subset T_x S^n$ (up to $\pm 1$) such that $\nu(x) \perp x$ and $\nu(x) \perp T_x \Sigma$. Also noting that $x \perp T_x S^n$, we know span$\{x,T_x \Sigma, \nu(x) \}=\R^{n+1}$. Thus $\Sigma$ satisfies Definition \ref{nond1}(i).
   By Remark \ref{sec}, we deduce that $\Sigma$ also Definition \ref{nond1}(ii). Thus Theorem \ref{mainproj2} implies Theorem \ref{mainproj}.
\end{proof}

We are left to show Theorem \ref{mainproj2}. We make some further reductions.

\begin{remark} \label{barkz}
 By Lemma \ref{equiv1}, if $\Sigma$ satisfies Definition \ref{nond1}(ii), we can without loss of generality assume all the principal curvatures $\kappa_i(x)$ are positive for all $x \in \Sigma$, $i=1, \cdots,n-1$. Finally, by the $C^2$-regularity of $\Sigma$, we know the $n-1$ principal curvatures $\kappa_i \in C^0(\Sigma, \mathbb{R_+})$. Thus there exists $\kappa \ge 1$ such that
 \begin{equation}\label{curcon}
   0<\kappa^{-1} \le \min_{x \in \Sigma}\{\kappa_1(x), \cdots, \kappa_{n-1}(x)\} \le \max_{x \in \Sigma}\{\kappa_1(x), \cdots, \kappa_{n-1}(x)\} \le \kappa < \infty.
 \end{equation}
\end{remark}


\begin{lemma} \label{cineprop}
  Let $Z$ be a subset in $B(0,1)\subset \mathbb{R}^{n+1}$ and $\Sigma : [0,1]^{n-1} \to \R^{n+1}$ be as in Theorem \ref{mainproj2}. Then  $F:=\{f_{z} \in C^2( [0,1]^{n-1}) : f_{z}(x):=z \cdot \Sigma(x) \}_{z \in Z}$ is a cinematic family with the same doubling constant as $\mathbb{R}^{n+1}$. In particular, $F$ is biLipschitz homeomorphic to $Z$.
\end{lemma}

\begin{proof}
First we show (1) in Definition \ref{cine}.
For any $x \in \Sigma$,
\begin{equation}\label{cine1}
  \sup_{f_{z} \in F }\|f_{z}\|_{C^2([0,1]^{n-1})} \le \|\Sigma\|_{C^2([0,1]^{n-1})} \max_{z \in Z}|z| \lesssim 1
\end{equation}
 where the implicit constant only depends on $\Sigma$. Thus (1) holds. Also for (4), the modulus of continuity $\az$ can be chosen as the one of $|\nabla^2 \Sigma|$.

Fix $y,z \in Z$. We show (3) in Definition \ref{cine}. Similar as \eqref{cine1}, we have
\begin{equation}\label{dou1}
  \|f_y-f_{z}\|_{C^2([0,1]^{n-1})} \lesssim |y-z|.
\end{equation}
Write
$h:= f_y-f_{z}$. We prove (3) by showing
\begin{equation}\label{cine2}
  \inf_{x \in [0,1]^{n-1}, \xi \in S^{n-2}} \{ |h(x)| + |\nabla h(x)| + |\nabla_\xi \nabla_\xi h (x)| \} \gtrsim  |y-z|.
\end{equation}
To apply the curvature condition \eqref{curcon},
define $\tilde h  \in C^2(\Sigma)$ by $\tilde h(w):= (y-z) \cdot w$, $w \in \Sigma$. Thus $h=\tilde h \circ \Sigma$. For each $\xi \in S^{n-2}$, let $\xi_\ast:= \nabla \Sigma(x) \cdot \xi \in T_{\Sigma(x)} \Sigma$. We know for all $x \in [0,1]^{n-1}$ and $\xi \in S^{n-2}$,
\begin{equation}\label{rel3}
  \langle\nabla h (x), \xi \rangle = \langle\nabla \tilde h (\Sigma(x)), \xi_\ast \rangle
\end{equation}
and
\begin{equation}\label{rel2}
  \nabla_\xi \nabla_\xi h (x) = \nabla_{\xi_\ast} \nabla_{\xi_\ast}\tilde h (\Sigma(x)) + \langle\nabla\tilde h (\Sigma(x)), \nabla^2\Sigma(x) (\xi,\xi)  \rangle
\end{equation}
where $\nabla^2 \Sigma(x)$ is a linear map from $T_x[0,1]^{n-1} \times T_x[0,1]^{n-1} $ to $T_w \Sigma$.
In particular, noting that $|\nabla \Sigma|$ is nowhere vanishing (since $\Sigma$ is a diffeomorphism), \eqref{rel3} implies that
\begin{equation}\label{rel4}
  |\nabla h (x)| \sim  |\nabla \tilde h (\Sigma(x))|, \quad x \in [0,1]^{n-1}
\end{equation}
where the implicit constant depends on $\max_{x \in [0,1]^{n-1}}\{|\nabla \Sigma(x)|\}$ and $\min_{x \in [0,1]^{n-1}}\{|\nabla \Sigma(x)|\}$.

Fix $x \in [0,1]^{n-1}$ and $\xi \in S^{n-2}$ in the following and let $w : =\Sigma(x)$. Also, let $ \{e_1, \cdots , e_{n-1} \}$  $\subset T \Sigma= \Sigma \times \R^{n-1}$ be an orthonormal basis on each point in $\Sigma$ and $e_1(w) = \xi_\ast$.
Noting that
$\tilde h (w)= (y-z) \cdot w$, we have
$\nabla_{e_i} \tilde h (w)= (y-z) \cdot \gamma_i '(0)$ where $\gamma_i :[-\epsilon,\epsilon] \to \Sigma$ is a $C^2$-curve
parametrized by arc-length such that $\gamma_i(0)=w$ and $\gamma_i'(0)=e_i$.
Thus
\begin{equation}\label{gradh}
  \nabla \tilde h (w)= \sum_{i=1}^{n-1} [(y-z) \cdot \gamma_i '(0)]e_i = \sum_{i=1}^{n-1} [(y-z) \cdot e_i]e_i.
\end{equation}
 Since $\Sigma$ satisfies Definition \ref{nond1}(i), we know there exists a $C^2$ vector field $\nu$ such that
\begin{equation}\label{decomp10}
  \nu \perp w, \ \nu \perp T_{w} \Sigma \mbox { and } \mbox{ span } \{w , T_{w} \Sigma, \nu\} = \mathbb{R}^{n+1}.
\end{equation}
Furthermore, we note that \eqref{decomp1} implies $w \ne 0$ for all $w \in \Sigma$, i.e. $\Sigma$ does not contain the origin.
Denoting by $e_0:= \frac{w}{|w|}= \frac{\Sigma(x)}{|\Sigma(x)|}$, we have
\begin{equation}\label{decomp1}
  \nu \perp e_0, \ \nu \perp T_{w} \Sigma \mbox { and } \mbox{ span } \{e_0 , T_{w} \Sigma, \nu\} = \mathbb{R}^{n+1}.
\end{equation}
Thus $y-z$ can be uniquely written as
\begin{equation*}
  y-z =  \sum_{i=0}^{n-1}[(y-z) \cdot e_i] e_i + [(y-z) \cdot \nu] \nu.
\end{equation*}
Let $I_\Sigma$ be the first fundamental form of $\Sigma$ and
$|I_\Sigma| = \sum_{i,j,k=1}^{n-1} |\Gamma_{ij}^k|$
where $\Gamma_{ij}^k= \langle \nabla_{e_j}e_i,e_k \rangle$ are the Christoffel symbols, $1 \le i \le n-1$.
Define
\begin{equation}\label{m}
  m:= \max_{p \in \Sigma} \{\langle e_0,\nabla_{e_1}e_1 \rangle|_{p}\}.
\end{equation}
Recall $\kz$ in \eqref{curcon} in  Remark \ref{barkz} and denote
$$   M : =  \frac{\min\{\kz^{-1},  \min_{p \in \Sigma}\{|p|\},1\}}{10n \max\{|I_\Sigma|,m , \|\Sigma\|_{C^2([0,1]^{n-1},\R^{n+1})},1 \}} \in (0,\frac{1}{10n}]$$
where the fact $M>0$ is because \eqref{decomp1} implies $\min_{p \in \Sigma}\{|p|\} >0$.
In addition, we notice that $|p| \ge M$ for all $p \in \Sigma$.

We consider the following three cases separately.

\medskip
\emph{Case 1.}
If $|(y-z) \cdot e_0| \ge M |y-z|$, then by $\tilde h (w)= (y-z) \cdot (|w|e_0)$,
we deduce that
$$ |h(x)| = |\tilde h (w)| = |w| |(y-z) \cdot e_0| \ge |w|M|y-z| \ge M^2|y-z| \sim |y-z|.$$
Thus \eqref{cine2} holds.

\medskip
\emph{Case 2.}
If $|(y-z) \cdot e_i| \ge M |y-z|$ for some $i=1,\cdots, n-1$, then by \eqref{rel4} and \eqref{gradh},
$$ |\nabla h(x)| \sim |\nabla \tilde h(w)| \ge  |(y-z) \cdot e_i| \ge M |y-z| \gtrsim |y-z|,$$
which gives \eqref{cine2}.

\medskip
\emph{Case 3.}
$|(y-z) \cdot e_i| < M |y-z| \le (10n)^{-1}|y-z| $ for all $i=0,\cdots, n-1$, which implies
$|(y-z) \cdot \nu| \ge \frac12 |y-z|$. By direct computation, we have
\begin{align*}
  \nabla_{e_i}\nabla_{e_i} \tilde h(w) & = \nabla_{e_i}[(y-z)\cdot \gamma'_i(0)] = (y-z)\cdot \nabla_{e_i}e_i\\
   & = \sum_{j=0}^{n-1}[(y-z) \cdot e_j]\langle e_j,\nabla_{e_i}e_i \rangle + [(y-z) \cdot \nu]\langle \nu,\nabla_{e_i}e_i \rangle.
\end{align*}
Recalling that \eqref{m} and noting that $M \le \frac{\kz^{-1}}{10n m}$, we know
\begin{equation}\label{extnormal1}
  |\langle e_0,\nabla_{e_1}e_1 \rangle|M \le m \frac{\kz^{-1}}{10n m} \le \frac{\kz^{-1}}{10n}.
\end{equation}
Moreover, for $j=1,\cdots, n-1$, noting that $M \le \frac{\kz^{-1}}{10n |I_\Sigma|}$, we have
\begin{equation}\label{extnormal3}
   |\langle e_j,\nabla_{e_i}e_i \rangle| M =  |\Gamma_{ji}^i| M \le   \frac{\kz^{-1}}{10n}.
\end{equation}
By the definition of of principal curvature of $\Sigma$ as well as recalling \eqref{curcon},
\begin{equation}\label{extnormal2}
  \langle \nu ,\nabla_{e_1}e_1 \rangle = -\langle \nabla_{e_1}\nu,e_1 \rangle \ge \kappa^{-1} .
\end{equation}
As a result, recalling $\xi_\ast =e_1(w)$ and combining \eqref{extnormal1}, \eqref{extnormal3} and \eqref{extnormal2}, we obtain
\begin{align} \label{unifconv}
  |\nabla_{\xi_\ast}\nabla_{\xi_\ast} \tilde h (w)| & = \left|\sum_{j=0}^{n-1}[(y-z) \cdot e_j]\langle e_j,\nabla_{\xi_\ast}\xi_\ast \rangle + [(y-z) \cdot \nu]\langle \nu,\nabla_{\xi_\ast}\xi_\ast \rangle \right| \notag\\
   & \ge|[(y-z) \cdot \nu]\langle \nu,\nabla_{\xi_\ast}\xi_\ast \rangle| - \left|\sum_{j=0}^{n-1}[(y-z) \cdot e_j]\langle e_j,\nabla_{\xi_\ast}\xi_\ast \rangle \right| \notag\\
   & \ge \frac12|y-z|\kz^{-1} - |y-z|M \sum_{j=0}^{n-1} |\langle e_j,\nabla_{\xi_\ast}\xi_\ast \rangle| \notag\\
   & \ge \frac12|y-z| \kz^{-1} -\frac1{10}|y-z|\kz^{-1} \notag\\
   & \ge \frac2{5\kz} |y-z|.
\end{align}
Recalling that $\nabla^2 \Sigma(x)$ is a linear map from $T_x[0,1]^{n-1} \times T_x[0,1]^{n-1} $ to $T_w \Sigma$, we know $\nabla^2\Sigma(x) (\xi,\xi) \in T_w \Sigma$ and thus $\nabla^2\Sigma(x) (\xi,\xi) \in$ span$\{e_1, \cdots,e_{n-1}\}$. Combining \eqref{gradh} and the standing assumption in \emph{Case 3}, we deduce that
\begin{align*}
  |\langle\nabla\tilde h (\Sigma(x)), \nabla^2\Sigma(x) (\xi,\xi)  \rangle | & \le |\nabla^2\Sigma(x)| \sum_{i=1}^{n-1}|(y-z) \cdot e_i|  \\
   & < \|\Sigma\|_{C^2([0,1]^{n-1},\R^{n+1})} nM |y-z| \\
   & \le  \frac{1}{10\kz}|y-z|.
\end{align*}
Consequently, using \eqref{rel2}, we arrive at
$$ |\nabla_\xi \nabla_\xi h (x) | \ge |\nabla_{\xi_\ast}\nabla_{\xi_\ast} \tilde h (w)| -  |\langle\nabla\tilde h (\Sigma(x)), \nabla^2\Sigma(x) (\xi,\xi)  \rangle | \gtrsim  |y-z|, $$
which gives \eqref{cine2}.
Combining \emph{Cases 1-3}, (3) in Definition \ref{cine} holds.

Finally, we show (2) in Definition \ref{cine}. Indeed, combining \eqref{dou1} and \eqref{cine2}, we deduce that
\begin{equation*}
 |y-z| \lesssim \inf_{x \in [0,1]^{n-1}, \xi \in S^{n-2}} \{ |h(x)| + |\nabla h(x)| + |\nabla_\xi \nabla_\xi h (x)| \} \le  \|f_y-f_{z}\|_{C^2([0,1]^{n-1})} \lesssim |y-z|,
\end{equation*}
which implies that (2) in Definition \ref{cine} holds and
$F$ is a doubling space with the same doubling constant as $\mathbb{R}^{n+1}$. In particular, $F$ is biLipschitz homeomorphic to $Z$.

The proof is complete.
\end{proof}

We are in the position to show Theorem \ref{mainproj2}. The proof is similar to \cite[Section 3.2]{2021arXiv210603338O}.

\begin{proof}[Proof of Theorem \ref{mainproj2}]
  Let $\Hd Z =t$. Without loss of generality, we can assume $t \in (0,1]$. It suffices to show \eqref{Kauf} for all $0<s <t$.
  Consider the family
  $F \subset C^{2}([0,1]^{n-1})$,
  $$F:=\{ f_z=\langle \Sigma(x),z \rangle  \}_{z\in Z}.$$
  Applying Lemma \ref{cineprop}, we know $F$ is a cinematic family with cinematic constant $K$, doubling constant $D$ and modulus of continuity $\az$ only
  depending on $\Sigma$ and $n$.

  Define $L:= \max_{f_z \in F}\{\|f_z\|_{L^\fz([0,1]^{n-1})}\} \le K$. Then one can check that the family
  $$ \bar F : = \{\frac1{2L}(f_z + L) : f_z \in F\} \subset C^{2}([0,1]^{n-1}, [0,1])$$
  is a cinematic family with  cinematic constant $\bar K \sim K$, doubling constant $\bar D \sim D$ and modulus of continuity $\bar \az \sim \az$. Thus without loss of generality, we can assume $ F  \subset C^{2}([0,1]^{n-1}, [0,1])$.


%

  We prove by contradiction. Assume there exists $0<s< t=\Hd Z$ and $X_s \subset [0,1]^{n-1}$ such that
  $$ \Hd X_s > n-2+s $$
  and
  \begin{equation}\label{sdim}
    \Hd(\Sigma(x) \cdot Z ) \le s, \quad x \in X_s.
  \end{equation}
   We further formulate this statement in a discrete way.
  Fix $s'<t$ such that $n-2+s'\in (n-2+s, \Hd X_s)$.
  Then $H: = \calH^{s'}_\infty(Z) >0$.
  Moreover, we have
  $$\calH^{n-2+s'}_\infty( X_s ) >0 $$
  and
  $$ \calH^{s'}(\Sigma(x) \cdot Z  ) =0, \quad x \in X_s.   $$
  We apply Theorem \ref{thmconfig} with $s=t=s'$. Then we get a parameter $\ez(s',n,K')$ and $\dz_0=\dz_0(\ez)$ such that Theorem \ref{thmconfig} holds for all $\dz< \dz_0$. Moreover, by decreasing $\dz_0$, we fix $\ez>0$ such that
  \begin{equation}\label{choiceez}
   \ez < \min\{\ez(s',n,K'), \frac{s'-s}{48} \}, \mbox{ and } \dz_0^\ez <  \min\{[\log(1/\dz_0)]^{-2},C^{-1}\}
  \end{equation}
  where $C>1$ is a constant to be determined.

   Fix $k_0 \in\N$ such that $2^{-k_0} < \dz_0$. For each $x \in X_s$, let $\calU_x = \{ I_x^i\}_{i \in\mathcal{I}_x}$ be an arbitrary cover of $\Sigma(x) \cdot Z $
by dyadic intervals with length $\le 2^{-k_0}$ and
$(\Sigma(x) \cdot Z ) \cap I_x^i \ne \emptyset$ for all $i \in \mathcal{I}_x$.
We know
$$    \sum_{i \in \mathcal{I}_x} | I_x^i|^s \le 1.     $$
Write $\mathcal{I}_x^k:= \{ i \in \mathcal{I}_x : | I_x^i|\in [2^{-k},2^{-(k-1)})\}$ for all $k \ge k_0$ and
\begin{equation}\label{zxk}
  Z_x^k : = \{z \in Z: \langle \Sigma(x),z \rangle \in \bigcup_{i \in \mathcal{I}_x^k}I_x^i \}.
\end{equation}
Thus $\{Z_x^k\}_{k \ge k_0}$ forms a cover of $Z$ for all $x \in X_s$.
Recall $H= \calH^{s'}_\infty(Z) >0$. Thus for each $x \in X_s$, by pigeonhole principle, there exists $k_x \ge k_0$ such that
$$    \calH^{s'}_\infty(Z_x^{k_x}) > k_x^{-2} \mbox{ and } |\mathcal{I}_x^{k_x}|  \overset{\eqref{sdim}}{\le} 2^{sk_x}.  $$
By using a second pigeonhole principle, we can find $k_1 \ge k_0$, and a subset $\bar X_s \subset X_s$ such that
  $$  \calH^{n-2+s'}_\infty(\bar X_s)   > k_1^{-2} \mbox{ and } k_x=k_1 \quad x \in \bar X_s.  $$
Let $\dz: = 2^{-k_1}$.
  By \cite[Lemma 2.7]{2021arXiv210603338O}, we know that there exists a $\delta$-separated $(\delta,n-2+s',C)$-set $Y \subset \bar X_s$ satisfying
  \begin{equation}\label{cardy}
    C^{-1}\delta^{-(n-2+s')} \leq |Y| \leq \delta^{-(n-2+s')}
  \end{equation}
  where $C$ is the constant to be determined in \eqref{choiceez}.
  As a result, we have
  \begin{equation}\label{upb2}
    \calH^{s'}_\infty(Z_x^{k_1}) > k_1^{-2} = [\log(1/\dz)]^{-2}  \mbox{ and } |\mathcal{I}_x^{k_1}|  \le \dz^{-s} \quad x \in Y.
  \end{equation}

  Since $Z_x^{k_1} \subset Z$, again
 by \cite[Lemma 2.7]{2021arXiv210603338O}, we can further find a $(n+1)\delta$-separated $(\delta,s',C[\log(1/\dz)]^{-2})$-set $W_x \subset Z$ satisfying $C^{-1}\delta^{-s'}[\log(1/\dz)]^{2} \leq |W_x| \leq \delta^{-s'}$. Since $W_x \subset [0,1]^{n+1}$ is $(n+1)\delta$-separated,  the family of dyadic $\delta$-cubes $\{Q\}$ in $[0,1]^{n+1}$ such that $Q \cap W_x \ne \emptyset$ has the same cardinality as $W_x$. So we consider $W_x$ as a set of dyadic $\delta$-cubes $[0,1]^{n+1}$ in the following.

 Let $W:= \cup_{x \in Y} W_x$. Thus $|W| \ge |W_x| \ge \dz^\ez\delta^{-s'}$ and $W$ is a family of dyadic $\dz$-cubes $Q$ with
 \begin{equation}\label{cardw}
   \calH^{s'}_\infty(W) = \calH^{s'}(W) \gtrsim \dz^\ez\delta^{-s'}(\diam Q)^s \gtrsim \dz^\ez.
 \end{equation}
Noting that
$$  \sum_{w \in W} |\{x\in Y:  w \in W_x\}| =  \sum_{x \in Y} | W_x|   \ge C^{-1}\delta^{-s'}[\log(1/\dz)]^{2}|Y| \overset{\eqref{choiceez}}{\gtrsim} \dz^\ez\delta^{-s'}|Y|,$$
we infer that  there exists a subset $\bar W \subset W$ with $|\bar W| \ge \frac12 \dz^\ez|W|$ such that
$$  |\{x\in Y:  w \in W_x\}| \ge \frac12 \dz^\ez|Y|, \quad  w \in \bar W . $$
By \eqref{cardw}, we know $\calH^{s'}_\infty(\bar W) \gtrsim \dz^\ez$. Thus we can find a $(\delta,s',\dz^{-3\ez})$-set of $\bar W$ still denoting by $\bar W$ with
 $\dz^{3\ez}\delta^{-s'} \leq |\bar W| \leq \delta^{-s'}$.

 To apply Theorem \ref{thmconfig}, we construct a $(\delta,s',s',\dz^{-3\ez})_{n-1}$-configuration. For each $w \in \bar W$ which is considered as a $\dz$-cube, we can find an element  $z_w \in Z \cap w$ since each $\dz$-cube $w$ has nonempty intersection with $Z$ by construction. And now we consider $\bar W$ as the set $\{z_w\}$ and we still write the element in  $\bar W$ as $w$ in the following.

 Define
 $$   E_w:=\bigcup_{x \in Y} (x, f_w(x)) = \bigcup_{x \in Y} (x, \langle \Sigma(x),w \rangle ) \subset [0,1]^{n} \quad w \in W.$$
 Since $Y$ is a $(\delta,n-2+s',C)$-set in $[0,1]^{n-1}$, $E_w$ is a $(\delta,n-2+s',C)$-set in $[0,1]^{n}$ with $|E_w|=|Y|$.
 Define
 $$ \Omega:= \bigcup_{w \in \bar W} (f_w,E_w) .$$
 Then $\Omega$ is a $(\delta,s',s',\dz^{-3\ez},|Y|)_{n-1}$-configuration.
 Let $\calE := \bigcup_{w \in \bar W} E_w$. Applying Theorem \ref{thmconfig} and using \eqref{cardy}, we know
 \begin{equation}\label{lob1}
   |\calE|_\dz \ge \dz^{48\ez+2-2s'} .
 \end{equation}

On the other hand,
recalling $\bar W \subset W \subset \cup_{x \in Y} W_x \subset \cup_{x \in Y} Z_x^{k_1}$,
we have
$$  \calE \subset \bigcup_{x \in Y} \{x\} \times \langle \Sigma(x),\cup_{x \in Y} Z_x^{k_1} \rangle.  $$
Recalling the definition of $Z_x^{k_1}$ in \eqref{zxk} and combining \eqref{upb2},
we deduce
$$  |\calE|_\dz \le \sum_{x \in Y} |Z_x^{k_1}|_\dz \le |Y|\dz^{-s}\le \dz^{2-s-s'} .  $$
This contradicts to \eqref{lob1}, since $\ez < 48^{-1}(s'-s)$ by our choice of $\ez$ in \eqref{choiceez}. The proof is complete.
\end{proof}

Now, we give a new proof of Theorem \ref{mattila} in $\R^n$ for all $n \ge 3$.

\begin{proof}[Proof of Theorem \ref{mattila} with $n \ge 3$]
  Let $\Hd Z =t$. Without loss of generality, we can assume $t \in (0,1]$. Since we can partition $\Sigma$ into pieces with each piece diffeomorphic to $[0,1]^{n-1}$, we can prove Theorem \ref{mattila} for each piece. Thus we may assume $\Sigma:[0,1]^{n-1} \to S^{n-1}$ is a $C^\fz$ diffeomorphism from $[0,1]^{n-1}$ to $\Sigma([0,1]^{n-1}) \subset S^{n-1}$ with $\|\Sigma\|_{C^1([0,1]^{n-1})} \lesssim 1$.

  Consider the family $F \subset C^{2}([0,1]^{n-1})$, $F:=\{ f_z=\langle \Sigma(x),z \rangle  \}_{z\in Z}$. For the same reason as in the proof of Theorem \ref{mainproj}, without loss of generality, we can assume $ F  \subset C^{2}([0,1]^{n-1}, [0,1])$. 
  If we can show that $F$ is a cinematic family with cinematic constant $K$ and doubling constant $D$ only
  depending on $\Sigma$ and $n$, then using the same proof for Theorem \ref{mainproj}, we can proof Theorem \ref{fundproj} as desired. Hence, it suffices to show the family $F$ enjoys Definition \ref{cine}.

  Indeed, parallel to the proof of Lemma \ref{cineprop}, one could obtain that $F$ satisfies Definition \ref{cine} (1), (2) and (4). We only need to check that $F$ satisfies Definition \ref{cine} (3). The verification is also similar to (and easier than) that  in Lemma \ref{cineprop}.

  Fix $y,z \in Z$. To see (3) in Definition \ref{cine}.
Write
$h:= f_y-f_{z}$. We prove (3) by showing
\begin{equation}\label{cine21}
  \inf_{x \in [0,1]^{n-1}, \xi \in S^{n-2}} \{ |h(x)| + |\nabla h(x)|  \} \gtrsim  |y-z|.
\end{equation}
Note that the difference here is that there is no second order derivative term in \eqref{cine21} compared with \eqref{cine} and we can prove (3) even by removing the second order derivative term. This is because $\Sigma$ has same dimension as $S^{n-1}$ while in Lemma \ref{cineprop}, $\Sigma$ has codimension $1$ in $S^{n}$.

Define $\tilde h  \in C^2(\Sigma)$ by $\tilde h(w):= (y-z) \cdot w$, $w \in \Sigma$. Fix $x \in [0,1]^{n-1}$ and $\xi \in S^{n-2}$ in the following and let $w : =\Sigma(x)$.
Using the same argument and same notations as in the proof of Lemma \ref{cineprop}, we obtain
\begin{equation}\label{gradh1}
  \nabla \tilde h (w)= \sum_{i=1}^{n-1} [(y-z) \cdot \gamma_i '(0)]e_i = \sum_{i=1}^{n-1} [(y-z) \cdot e_i]e_i
\end{equation}
and
\begin{equation}\label{rel5}
  |\nabla h(x)| \sim |\nabla \tilde h(\Sigma(x))|= |\nabla \tilde h(w)|, \quad x\in[0,1]^{n-1}.
\end{equation}
Noting that $\Sigma \subset S^{n-1}$ is an $(n-1)$-dimensional manifold and $e_0(w) := w= \Sigma (x)$, we deduce that
\begin{equation*}
  e_0 \perp T_{w} \Sigma \mbox { and } \mbox{ span } \{e_0 , T_{w} \Sigma\} = \mathbb{R}^{n}.
\end{equation*}
Thus $y-z$ can be uniquely written as
\begin{equation*}
  y-z =  \sum_{i=0}^{n-1}[(y-z) \cdot e_i] e_i .
\end{equation*}
This implies that there exists at least one $i=0,\cdots,n-1$ such that $|(y-z) \cdot e_i| \ge \frac{1}{n} |y-z|$.

We consider the following two cases separately.

\medskip
\emph{Case 1.}
If $|(y-z) \cdot e_0| \ge \frac{1}{n} |y-z|$, then by $\tilde h (w)= (y-z) \cdot (|w|e_0)$,
we deduce that
$$ |h(x)| = |\tilde h (w)| = |w| |(y-z) \cdot e_0| \ge |w|\frac{1}{n}|y-z| \ge \frac{1}{n}|y-z| \sim |y-z|$$
where we recall $w \in \Sigma \subset S^{n-1}$ and hence $|w|=1$.
Thus \eqref{cine21} holds.

\medskip
\emph{Case 2.}
If $|(y-z) \cdot e_i| \ge \frac{1}{n} |y-z|$ for some $i=1,\cdots, n-1$, then by \eqref{gradh1} and \eqref{rel5},
$$ |\nabla h(x)| \sim |\nabla \tilde h(w)| \ge  |(y-z) \cdot e_i| \ge \frac{1}{n} |y-z| \gtrsim |y-z|,$$
which gives \eqref{cine21}.
The proof is complete.
\end{proof}

\begin{remark}
  We end the paper by remarking that Theorem \ref{mainproj2} implies \cite[Theorem 1.1]{ohm2023projection} in $\R^{n+1}$ for $n\ge 3$. Indeed, Theorem 1.1 in \cite{ohm2023projection} is a special case of Theorem \ref{mainproj2} where the intersection of $(n-1)$-dim family of lines and the set $H=\{x=(x_1,\cdots,x_{n+1})\in \R^{n+1}: x_1=1\}=\R^{n} \subset \R^{n+1}$ forms the graph $\Sigma_f \subset H$ of a function $f \in C^\infty(\R^{n-1})$ in the following type
  $$   f(y) = (Ly)^{\top} A (Ly), \quad y \in \R^{n-1}    $$
  where $L: \R^{n-1} \to \R^{n-1}$ is an isomorphism of $\R^{n-1}$, $A$ is a positive definite $(n-1)\times (n-1)$ matrix and $z^\top$ is the transpose of $z \in \R^{n-1}$.

  To see Theorem \ref{mainproj2} implies \cite[Theorem 1.1]{ohm2023projection}, we first show the graph $\Sigma_f$ satisfies Definition \ref{nond1}(1). Indeed, noting that for every $x \in \Sigma_f \subset H$, $x$ is not parallel to $T_x\Sigma_f$, thus span$\{x,T_x\Sigma_f\}$ is $n$-dimensional. We can take $\nu(x)$ as the unit vector in $($span$\{x,T_x\Sigma_f\})^\perp$ (up to $\pm 1$) such that $\nu \in C^2(\Sigma_f,\R^{n+1})$.
  Hence Definition \ref{nond1}(1) holds for $\Sigma_f$.

  To see $\Sigma_f$ satisfies Definition \ref{nond1}(2), noticing that $f$ is convex at all $x \in \R^{n-1}$, thus the graph $\Sigma_f$ has positive sectional curvature.
  Also noting that
  $\Sigma_f$ is a hypersurface in $H$ and $H$ has constant sectional curvature $0$, by Gauss' formula, we know
  $$ 0< K_{\xi_i,\xi_j}(x) = \kz_{i}(x) \kz_{j}(x) ,\quad x \in \Sigma_f  $$
  where $\{\xi_i\}_{1 \le i\le n-1} \subset T_x\Sigma$ are the $n-1$ principal directions,
  $K_{\xi_i,\xi_j}(x)$ is the sectional curvature of $\Sigma_f$ at $x$ spanned by $\xi_i \in T_x\Sigma_f$ and $\kz_{i}(x)$ is the principal curvature in the direction $\xi_i$.
  Thus all principal curvatures of $\Sigma_f$ must have same sign at all $x\in \R^{n-1}$. Applying Lemma \ref{equiv1}, we get the desired conclusion.
\end{remark}

\bibliographystyle{plain}
\bibliography{references.bib}

\end{document}